\documentclass{amsart}

\usepackage{amsmath,amsfonts,amsthm,amssymb,amscd}
\usepackage{appendix}
\usepackage{cite}
\usepackage{geometry}
\usepackage{graphicx}
\usepackage{hyperref}
\hypersetup{CJKbookmarks}
\usepackage{setspace}
\usepackage{url}

\newtheorem{theorem}{Theorem}[section]

\newtheorem{lemma}[theorem]{Lemma}
\newtheorem{proposition}{Proposition}[section]
\newtheorem{corollary}{Corollary}[theorem]

\theoremstyle{definition}
\newtheorem{definition}{Definition}
\newtheorem{remark}{Remark}
\numberwithin{equation}{section}
\newcommand{\vertiii}[1]{{\left\vert\kern-0.25ex\left\vert\kern-0.25ex\left\vert #1 \right\vert\kern-0.25ex\right\vert\kern-0.25ex\right\vert}}

\author{Garrett Heller}
\address{Thomas Jefferson High School for Science and Technology, 6560 Braddock Rd, Alexandria, VA}
\email{2022gheller@tjhsst.edu}

\author{Chengyang Shao}
\address{Massachusetts Institute of Technology, 77 Mass Ave, Cambridge, MA 02139}
\email{shaoc@mit.edu}

\title{Strichartz and Multi-linear Estimates for the One-dimensional Periodic Dysthe equation}

\begin{document}
\maketitle
\begin{spacing}{1.1}
\begin{abstract} 
\normalsize This paper presents Strichartz estimates for the linearized 1D periodic Dysthe equation on the torus, namely estimate of the $L^6_{x,t}(\mathbb{T}^2)$ norm of the solution in terms of the initial data, and estimate of the $L^4_{x,t}(\mathbb{T}^2)$ norm in terms of the Bourgain space norm. The paper also presents other results such as bilinear and trilinear estimates pertaining to local well-posedness of the 1-dimensional periodic Dysthe equation in a suitable Bourgain space, and ill-posedness results in Sobolev spaces.
\end{abstract}
\section{Introduction and Background}

\subsection{History of the Topic}
Rogue Waves are unusually large surface waves characterized by having at least double the significant wave height, or eight times the standard deviation of the surface elevation \cite{DystheKrogstadMuller2008}. They are capable of causing great destruction such as hitting ships or even buildings on shore \cite{DidenkulovaSlunyaevPellinovskyKharif2006}. In 2005 alone, Didenkulova et. al. reported three incidents of Rogue Waves hitting passenger ships, causing multiple injuries and ship damage, including a 15 meter wave striking the cruise ship ``Voyager" \cite{DidenkulovaSlunyaevPellinovskyKharif2006}. As such, researchers have focused heavily on Rogue Wave behavior and specifically the mathematical models of it in the past quarter century \cite{DystheKrogstadMuller2008}, \cite{FarazmandSapsis2017}. Among these models, the most well-known one is the Dysthe equation.

In 1979, Dysthe derived the multi non-linear Schrödinger equation (MNLS), now called the \emph{Dysthe equation}, from a hydrodynamical perspective \cite{Dysthe1979}. The Dysthe equation is at the center of heavy research today in fluid mechanics due to its relevancy in deep water Rogue Waves. Due to the complexity of the nonlinear terms in the Dysthe equation, these Rogue Waves are still not yet fully understood mathematically. In 2017, Farazmand and Sapsis conducted a numerical analysis of these Rogue Waves based on the Dysthe equation, in order to predict water wave and weather patterns \cite{FarazmandSapsis2017}. More numerical results would be aided by a better mathematical understanding of the Dysthe equation. In 1993, a landmark series of papers by Bourgain introduced new techniques on Fourier transform restriction norms, which could be used to obtain norm bounds on functions on a periodic setting \cite{Bourgain1993}. Bourgain obtained bounds on solutions to naturally arising Diophantine equations which were used to prove the well-posedness of the Korteweg de Vries (KdV) equation on $\mathbb{T} \times \mathbb{R}$ \cite{Bourgain1993}. Many of the methods later in Bourgain's paper were improved by the group Colliander, Keel, Staffilani, Takaoka and Tao, collectively known as the ``I-Team." In a 2006 work \cite{CKSTT} they formalized an approach based on multilinear estimates to achieve local well-posedness of the KdV equation on the torus.
    
In this paper, we shall extend the techniques of Bourgain, the ``I-Team," and others in the hopes of contributing to a local well-posedness proof on the periodic $\mathbb{T}^2$ for the Dysthe equation.

\subsection{Fourier Transform Notation}
If $f(x)$ is a function of the spatial variable $x$ only, then $\widehat{f}(n)$ denotes the Fourier transform of $f$ with respect to $x$. If $g(x,t)$ is a function of the spacetime variable $(x, t)$, then $\widehat{g}(n, \tau)$ denotes the Fourier transform of $g$ with respect to $(x,t)$. In this case, $(\mathcal{F}_x) g (n, t)$ denotes the Fourier transform of $g$ with respect to the spatial variable $x$. 

\subsection{1-Dimensional Dysthe Equation}
We shall state the 1-dimensional Dysthe equation derived in \cite{Dysthe1979}. The problem shall be framed as a periodic one, that is the space variable will be $x\in\mathbb{T}$. The periodic setting has been used for numerical studies of Rogue Waves in \cite{FarazmandSapsis2017}, while well-posedness in this setting has not been worked on in this area. The problem differs from that framework of the whole-space: in the whole-space setting, the Dysthe equation exhibits dispersive behaviors (see e.g. \cite{GrandeKurianskiStaffilani2020}), whereas solutions in the periodic setting do not decay. There are other tools one can use in the periodic framework, such as techniques from discrete mathematics. The paper of Grande, Kurianski and Staffilani \cite{GrandeKurianskiStaffilani2020} improved local well-posedness results in the whole-space setting; however, less results have been achieved in the periodic setting. 

In the following, we use $u^*$ to denote the complex conjugate of $u$. Consider the Cauchy problem for the 1-dimensional Dysthe equation:

\begin{equation}\label{1deq}\left\{
\begin{aligned} \partial_t{u} + &\frac{1}{2} \partial_x u + \frac{i}{8}\partial_{xx} u - \frac{1}{16}\partial_{xxx}{u} + \frac{i}{2}|u|^2u + \frac{3}{2}|u|^2 \partial_x u + \frac{1}{4}u^2\partial_x{u^*} -  \frac{i}{2}u|\partial_x||u|^2 =0 \\
u(x, 0) &=u_0(x)
\end{aligned}
\right.
\end{equation}

In the above equation, $u$ represents the complex-valued wave envelope, and it depends on a single space coordinate $x\in\mathbb{T}$, and the time coordinate $t$. This paper focuses on well-posedness in the torus $[0, 2\pi]_x \times [0, 32\pi]_t$. Under this setting, the wave envelope is then periodic with respect to both $x$ and $t$. 

Linearizing (\ref{1deq}) around the trivial solution $u\equiv0$, one obtains the linearized Dysthe equation:
\begin{equation}\label{1dlin}\left\{
\begin{aligned}
\partial_t u + &\frac{1}{2} \partial_x u + \frac{i}{8}\partial_{xx} u - \frac{1}{16}\partial_{xxx}{u} =0 \\
u(x, 0) &= u_0(x)
\end{aligned}
\right.
\end{equation}
for a given function $u_0$. 

Taking the Fourier transform with respect to $x$ to get rid of the partial derivatives in $x$, one finds the equation in the Fourier side is
$$
\partial_t (\mathcal{F}_x{u})(n,t) + \frac{i}{2}n (\mathcal{F}_x{u})(n,t) - \frac{i}{8} n^2 (\mathcal{F}_x{u})(n,t) + \frac{i}{16}n^3 (\mathcal{F}_x{u})(n,t) =0
$$
and $\widehat{u}(x, 0) = \widehat{u}_0 (x)$.
The former is a separable equation so integration yields
$$
(\mathcal{F}_x{u})(n,t) = e^{ip(n)t}\widehat{u}_0(n) 
$$
where $p$ is the dispersive relation given by 
$$
p(n) = -\frac{n^3}{16} + \frac{n^2}{8} - \frac{n}{2}.
$$

It turns out one can transform the torus $[0, 2\pi]_x \times [0, 32\pi]_t$ to the standard $\mathbb{T}^2 = [0, 2\pi]_x \times [0, 2\pi]_t$ by scaling our dispersive relation by $16$ to 
\begin{equation}\label{dispersive}
P(n) = n^3 - 2n^2 + 8n.
\end{equation}
From this point forward we shall consider (\ref{dispersive}) as our dispersive relation, and $[0, 2\pi]_x \times [0, 2\pi]_t$ as our spacetime. With this convention, one finds that the solution of the linearized equation (\ref{1dlin}) is
\begin{equation}\label{1dlinsol}
 u(x, t) = \sum_{n \in \mathbb{Z}}{e^{inx + iP(n)t} \widehat{u}_0(n)}.
\end{equation}

\begin{definition} \label{linprop}
Define the linear propagator $e^{it\mathcal{L}}$ such that 
$$
(e^{it\mathcal{L}}u)(x, t) 
= \sum_{n \in \mathbb{Z}}{e^{itP(n)}(\mathcal{F}_xu)(n,t)e^{inx}}.
$$
Notice the linear propogator acting on the initial state denotes the solution to the linearized Dysthe equation.
\end{definition}

Next define a nonlinear mapping $\mathcal{N}(u)$ by 
\begin{equation}\label{Nonlinear}
\mathcal{N}(u) = -\frac{i}{2}|u|^2u - \frac{3}{2}|u|^2 \partial_x u - \frac{1}{4}u^2\partial_x{u^*} + \frac{1}{2}iu|\partial_x||u|^2.
\end{equation}
Then by Duhamel's principle, the solution $u$ to the Dysthe equation (\ref{1deq}) for $|t|<1$ must satisfy the integral equation
\begin{equation}\label{duhamelrep}
u(t) = \eta(t)e^{it\mathcal{L}} u_0 + \eta(t)\int_{0}^{t}{d\tau e^{i(t-\tau)\mathcal{L}}\mathcal{N}(u(\tau))}
\end{equation}
where we introduced a smooth localization (bump) function $\eta$ supported by $[-2, 2]$ and $1$ on the interval $[-1, 1]$. This form is convenient for a few reasons: it reduces our initial value problem to a single equation, weakens the need for differentiability, and given sufficient bounds it can be used in a standard Banach's Contraction Mapping argument. 

\subsection{Main Results}
The main results of our paper are the following three theorems. The first one is the Strichartz estimate for the linearized 1D periodic Dysthe equation:

\begin{theorem} \label{str}
Let the solution to the linearized 1d periodic Dysthe initial value equation be $u(x, t),$ with $u(x, 0) = u_0(x)$ as the initial condition. Then $u$ satisfies the following Strichartz Inequality for every $\epsilon >0$:
$$
\|u\|_{L^6_{x, t}(\mathbb{T}^2)} \leq C_{\epsilon}\|u_0\|_{H^\epsilon},
$$
where $C_\epsilon$ is a constant depending only on $\epsilon$. A useful restatement of this theorem is that for any $\epsilon>0$
$$\|e^{it\mathcal{L}}f\|_{L^6_{x,t}(\mathbb{T}^2)} \lesssim \|f\|_{C_tH^{\epsilon}_x(\mathbb{T}^2)}.
$$
\end{theorem}

\begin{corollary}\label{lrhxstr}
Let $u$ be as in (\ref{1dlinsol}). For $r\geq 6$,
$$
\|u\|_{L^{r}_{x, t}} \leq C_{\epsilon, r} \|u_0\|_{H^{\frac{1}{4} - \frac{3}{2r} + \epsilon}}.
$$
\end{corollary}
Section \ref{2} will be devoted to the proof of Theorem \ref{str} and Corollary \ref{lrhxstr}. The method closely follows that of Bourgain \cite{Bourgain1993}.

\begin{theorem}\label{L4Ysb}
The second of our main results is known Bourgain's $L^4$ inequality for the Dysthe equation:
$$\|f\|_{L^4_{x,t}(\mathbb{T}^2)} \lesssim \|f\|_{X^{0, \frac{1}{3}}(\mathbb{T}^2)}$$
\end{theorem}
The definition of the Bourgain spaces $X^{s,b}$ and $Y^{s,b}$ corresponding to the dispersive relation $P(n)$ will be given in Section \ref{3} in this paper, and the whole section will be devoted to the proof of this theorem. The theorem was originally proved by Bougain in \cite{Bourgain1993} for the Airy equation. There are a few key differences with the proof for the linearized Dysthe equation. Section \ref{3} will focus on this proof and other linear inequalities.

The third of our main results is a bilinear estimate in the Bourgain space. Define $Z^{s, b} = X^{s, b} \cap Y^{s, b-\frac{1}{2}}$.
\begin{theorem}\label{bilinear}
For any $s\geq \frac{1}{2}$
$$
\|\mathbb{P}(u_1)\mathbb{P}(u_2)\|_{Z^{s, \frac{-1}{2}}} \lesssim \|u_1\|_{Z^{s-1, \frac{1}{2}}}\|u_2\|_{Z^{s-1, \frac{1}{3}}} + \|u_1\|_{Z^{s-1, \frac{1}{3}}}\|u_2\|_{Z^{s-1, \frac{1}{2}}}.
$$
Here the operator $\mathbb{P}$ is the orthogonal projection to the zero-mean-value subspace of $L^2_x(\mathbb{T})$.
\end{theorem}
Section \ref{4} will be devoted to the proof of Theorem \ref{bilinear} and other multilinear estimates. Such results were initially obtained by Bourgain in \cite{Bourgain1993} and later extended by Kenig-Ponce-Vega \cite{KenigPonceVega1996} and Colliander-Keel-Staffilani-Takaoka-Tao \cite{CKSTT} for the KdV equation.

\begin{theorem}\label{ill}
The initial value problem for the 1D periodic Dysthe equation is ill-posed in $H^s(\mathbb{T})$ for $s<0$.
\end{theorem}
The proof of Theorem \ref{ill} will be discussed at length in Section \ref{6}. Motivations for such ill-posedness and conjectures of stronger results can be found in Section \ref{5}. Similar results were shown recently in \cite{GrandeKurianskiStaffilani2020} for the Dysthe equation in different spatial dimensions and settings.

\subsection{Obstructions and open questions}
We did not manage to prove local well-posedness of the full nonlinear Dysthe equation (\ref{duhamelrep}), but we did obtain many important estimates that could help to show local well-posedness or other properties. 

Superficially, the 1D Dysthe equation looks similar to a KdV type equation
$$
\partial_tu-\partial_x^3u=\partial_x[F(u)].
$$
This suggests that three well-known methods for well-posedness of a KdV type equation, developed by mathematicians over the last half-century, could be applied: the viscosity method, the harmonic analysis approach, and the dynamical system approach. we will be focusing on the harmonic analysis approach, while still pointing out that all these three methods might fail for the periodic 1D Dysthe equation. The dynamical system approach is very specific to KdV so it is not applicable to the Dysthe equation, since the latter is not integrable. We will suggest in Section \ref{5} that the viscosity method fails for the Dysthe equation because the nonlocal term $i u|\partial_x||u|^2$ does not allow us to make an adequate a priori energy estimate.

The harmonic analysis approach is the most promising. However, the major difference between the KdV equation and 1D Dysthe equation in this case is that, there is a conservation of mean for the complex envelope for the former, which is the crucial property used by Bourgain \cite{Bourgain1993}, Kenig-Ponce-Vega \cite{KenigPonceVega1996} and Colliander-Keel-Staffilani-Takaoka-Tao \cite{CKSTT} to establish the necessary bilinear (trilinear) estimates to control the nonlinear term of the KdV type equations. Such arguments simply fail for the Dysthe equation. However, we was still able to obtain some crucial estimates on trilinear terms such as in lemma \ref{trilinear1}. We shall explain the details of this by the end of Section \ref{5}.

These results suggest that, in order to prove local well-posedness for the 1D periodic Dysthe equation, one might have to develop a completely new method.

\section{Estimates of the linearized equation}\label{2}

\subsection{Number of Resonant Frequencies}
To obtain Strichartz type inequalities for the one dimensional linearized Dysthe equation (\ref{1dlin}), some number-theoretical results will be needed.

\begin{definition}\label{rnj}
Let $n,j$ be integers and let $N$ be a natural number. Let 
$P(n)$ be given by (\ref{dispersive}),
i.e. the dispersive relation of the linear Dysthe equation. Set $r_{n,j}^{(N)}$ to be the number of ordered solutions $(n_1,n_2)$ over $\mathbb{Z}^2$ to the Diophantine equation
\begin{equation}\label{diophant}
P(n_1) + P(n_2) + P(n-n_1-n_2) = j
\end{equation}
with $|n_1|, |n_2| \leq  N$.
\end{definition}

\begin{lemma} \label{rnj1}
Let $N\in\mathbb{N}$ and $n,j\in\mathbb{Z}$ such that
$$
|n|\leq N^2,
\quad
|j|\leq N^6.
$$
Then
$$
r_{n,j}^{(N)} \leq C_{\epsilon} N^{\epsilon} 
\quad
\forall \epsilon >0,
$$
where the constant $C_{\epsilon}$ does not depend on $n,j, N$.
\end{lemma}
\begin{proof}
To simplify the notation, make the substitution $p = n_1 + n_2$, $q = n_1n_2$. Since $|n_1|,|n_2|\leq N$, there holds $|p| \leq 2N$, $|q| \leq N^2$. Notice that $n_1, n_2 \in \mathbb{Z} \implies p, q \in \mathbb{Z}$. Moreover, $p, q$ define the quadratic $f(X) = X^2 - pX + q = 0$ which has at most 2 integer roots for $X$ corresponding to $n_1, n_2$.

Now rewrite equation (\ref{diophant}) as 
$$
P(n_1) + P(n_2) + P(n-p) = j,
$$ 
using $n_1+n_2 =p$. Expanding the polynomial, it is found that
\begin{equation}\label{Temp2}
\begin{aligned}
-8(n_1+n_2) + 2(n_1^2 + n_2^2) -(n_1^3 + n_2^3) -8n + 8p +2(n^2 - 2np + p^2) - (n^3 -3n^2p + 3np^2 -p^3) = j,
\end{aligned}
\end{equation}
where we grouped the $n_1, n_2$ terms together. Using properties of fundamental symmetric polynomials, one finds $n_1+n_2 =p$, $n_1^2 + n_2^2 = p^2 - 2q$ and $n_1^3 + n_2^3 = p^3 - 3pq$. Substituting these into (\ref{Temp2}) and cancelling out the $8p$ gives
$$
2(p^2 -2q) -(p^3-3pq) -8n +2(n^2 - 2np + p^2) - (n^3 -3n^2p + 3np^2 -p^3) = j.
$$
Convert this to the following equation in $(p,q)$:
$$
\begin{aligned}
 4p^2 -3n p^2 +3pq + (-4n+3n^2)p -4q =  j+ n^3 -2n^2 + 8n.
 \end{aligned}
$$

Now set $j+n^3 -2n^2 +8n =k $ and factor the coefficient of $p$ to obtain
\begin{equation}\label{Temp1}
\begin{aligned}
(4-3n)p^2 - n(4-3n)p + 3pq-4q = k.
\end{aligned}
\end{equation}
Multiplying both sides by 9, one obtains
$$
9(4-3n)p^2 - 9n(4-3n)p + 27pq - 36q =9k.
$$
Subtracting $4(4-3n)^2$ from both sides, the left-hand-side can be factorized, and (\ref{Temp1}) becomes
\begin{equation}\label{Temp4}
    (3(4-3n)p + 9q + (4-3n)^2)(3p-4) = 9k-4(4-3n)^2.
\end{equation}
Let $9k-4(4-3n)^2 =  l$, and using the original definition of $k$ have
\begin{equation}\label{Temp5}
   l = 9j+9n^3 -18n^2 +72n  -4(4-3n)^2.
\end{equation}
Substituting $l$ into (\ref{Temp4}) yields:
\begin{equation}\label{Temp3}
    (3(4-3n)p + 9q + (4-3n)^2)(3p-4) = l.
\end{equation}

Consider the set $L$ of ordered pair of factors of $l$, that is  $L = \{(a, b) \in \mathbb{Z}^2 \mid ab = l\}$. For every $(a, b) \in L$, the linear system in $p$ and $q$
$$
\left\{
\begin{aligned}
&3(4-3n)p + 9q + (4-3n)^2= a\\
&3p-4 = b
\end{aligned}
\right.
$$
has at most one integral solution $(p,q)$ for every pair of factors of $(a, b)$. Since an ordered pair of factors is determined by the first number, the number of solutions $(p, q)$ to the diophantine equation (\ref{diophant}) is bounded by $|L|$, which is in turn bounded by the number of factors of $l$.

Recall from (\ref{Temp5}) $l$ is equal to $j$ plus a cubic in $n$. As such one can write $O(l) = O(j)+O(n^3) \implies l \lesssim N^6$ by the initial conditions on $n, j$. A well known number theoretic result on the divisor bound tells us that the number of factors of $l\lesssim N^6$ is bounded by (for some constants $c_{1,2}$) 
$$
\exp\left(c_1\frac{\log l}{\log \log l} \right)
\leq \exp\left({c_2\frac{\log N}{\log \log N} }\right),
$$
and the right-hand-side is bounded by $C_\epsilon N^{\epsilon}$ for any $\epsilon >0$. This completes the proof.
\end{proof}

Note in the above proof did not use the bounds on $n_1, n_2$. These are however used in the lemmas below. In the next two lemmas, shall achieve the desired bounds on $r_{n, j}^{(N)}$ when one of the conditions $|n| \lesssim N^2, |j| \lesssim N^6$ is violated. 
\begin{lemma} \label{rnj2}
Let $N\in\mathbb{N}$, and $n,j\in\mathbb{Z}$ such that
$$
|n|\geq N^2,
\quad
|j|\geq N^6.
$$
Then
$r_{n, j}^{(N)}\leq 3$ for sufficiently large $N$. 
\end{lemma}
\begin{proof}
Suppose there are multiple solutions $(n_1, n_2)$ and $(n_1', n_2')$ to (\ref{diophant}) satisfying $|n_1|, |n_2| \leq N$:
$$
P(n_1) + P(n_2) + P(n-n_1-n_2) = j,
\quad
P(n_1') + P(n_2') + P(n-n_1'-n_2') = j.
$$

Subtract these two equations with each other. The change in the left-hand-side is on the order of 
$$
\begin{aligned}
n_1^3 - & n_1'^3 + n_2^3 - n_2'^3 + (n-n_1-n_2)^3 - (n-n_1'-n_2')^3 
\\
&\sim n_1^3-n_1'^3 + n_2^3 - n_2'^3 + -3(n^2)(n_1+n_2 - n_1' - n_2'). 
\end{aligned}
$$
But notice the right-hand-side is constant, so the overall change in the left-hand-side must be zero. Assuming $n_1+n_2- n_1' - n_2'$ is nonzero, we have $$ |-3(n^2) (n_1+n_2-n_1'-n_2')| \gtrsim  N^4 .$$

However, $n_1^3 - n_1'^3 + n_2^3-n_2'^3 \leq 2(N^3 -(-N)^3) \lesssim N^3$ by the bounds on $n_1, n_2$. Therefore, the change in left-hand-side is dominated by the single term $-3n^3(n_1+n_2 -n_1'-n_2')$ for large $N$ assuming $(n_1+n_2 -n_1'-n_2')$. So thus the total change in the left hand side cannot be $0$ unless $n_1+n_2 =n_1' + n_2'$. As such, we see $n_1 + n_2$ must remain constant between different solutions.

Note that $n_1+n_2 =k$ defines a linear relationship between $n_1, n_2$. Then substitute $n_2 = k-n_1$ into the original equation. Since (\ref{diophant}) was a two-variable cubic now one obtains a single variable cubic in terms of $n_1$. Thus there are at most three values of $n_1$ satisfying these conditions. Once one finds $n_1$, $n_2 = k-n_1$ is fixed. So there are at most $3$ solutions obeying these bounds for sufficiently large $N$.
\end{proof}

\begin{lemma}\label{rnj3}
Let $N\in\mathbb{N}$ and $n,j\in\mathbb{Z}$ such that
$$
|n|\leq N^2,
\quad
|j|\geq N^6
$$
or 
$$
|n|\geq N^2,
\quad
|j|\leq N^6.
$$
Then
$r_{n, j}^{(N)}\leq 3$ for sufficiently large $N$. 
\end{lemma}
\begin{proof}
For sufficiently large $N$, first note that if $|n| \lesssim N^2$ and $|j| >> N^6$ then the right hand side of (\ref{diophant}) must grow faster than the left hand side, so there can be no solutions. A similar argument shows there are no solutions if $|n| >> N^2$ but $|j| \lesssim N^6$.

The remaining case is $|n| \geq cN^2$ and $|j| \geq N^6$ or $|n| \geq N^2$ and $|j| \geq c N^6$ for some constant $c<1$ independent of $N$. Then the proof for $r_{n, j}^{(N)} \leq 3$ is the same as the above lemma.
\end{proof}

Combining lemmas (\ref{rnj1}), (\ref{rnj2}), (\ref{rnj3}), we can now state
\begin{lemma}\label{diofinal}
If $r_{n, j}^{(N)}$ is as in definition \ref{rnj}, then given any $\epsilon>0$, there holds
$$
r_{n, j}^{(N)} \leq C_{\epsilon} N^{\epsilon}
$$ 
for any choice of $n, j$, with the constant $C_\epsilon$ depending on $\epsilon$ alone.
\end{lemma}

\subsection{Strichartz Bound}
Now turn to Strichartz type estimates for the linearized equation (\ref{1dlin}).

\begin{lemma}\label{partsum}
Suppose $u(x, t)$ is the solution to the linearized 1D Dysthe equation (\ref{1dlin}) with initial value $u(x, 0) = u_0(x)$. Let $S_Nu(x,t)=\sum_{|n|\leq N}\widehat u(n,t)e^{inx}$, the $N$th partial sum of the Fourier series expansion with respect to $x$. Then 
$$
\|S_Nu\|_{L^6(\mathbb{T}^2)} \lesssim N^{\epsilon}\|S_Nu_0\|_{L^2_x)}.
$$
\end{lemma}

\begin{proof}
Given integers $n,j$, let $\mathcal{A}_{n,j}$ be the set of all ordered integer tuples satisfying (\ref{diophant}), that is,  
$$
\mathcal{A}_{n,j}
= \left\{
(n_1, n_2) \in \mathbb{Z}^2 \mid P(n_1) + P(n_2) + P(n-n_1-n_2) = j.
\right\}
$$
Recall the representation for the solution $u$ given in (\ref{1dlinsol}). The integral should be as follows:
$$
\|S_Nu\|_{L^6(x,t)}^6 = \int_{[0,2\pi]}\left(\int_{[0,2\pi]}|S_Nu(x,t)|^6dt\right)dx,
$$
and one can estimate the integral through the Plancherel Theorem. Start with the equality
$$
\begin{aligned}
\|S_Nu\|^6_{L^6_{x,t}}
=\|(S_Nu)^3\|^2_{L^2(x, t)}.
\end{aligned}
$$
Compute
$$
\begin{aligned}
(S_Nu(x,t))^3 
&=\left(\sum_{|n|\leq N}e^{inx+iP(n)t}\widehat{u}_0(n)\right)^3\\
&=\sum_{|n_1|,|n_2|,|n_3| \leq N}
{e^{i(n_1+n_2+n_3)x}e^{i[P(n_1)+ P(n_2) +P(n_3)]t}
\widehat{u}_0(n_1)\widehat{u}_0(n_2)\widehat{u}_0(n_3)}.
\end{aligned}
$$
By the Plancherel Theorem, the $L^2_{x,t}$ norm over $\mathbb{T}^2$ of $(S_Nu(x,t))^3$ is thus
$$
\begin{aligned}
\|(S_Nu(x,t))^3\|_{L^2_{x,t}}^2 
&= \sum_{|n|\leq 3N, j \in \mathbb{Z}}\left| \sum_{\substack{n_1+n_2+n_3 = n,\\ P(n_1)+P(n_2)+P(n_3) = j \\ |n_1|, |n_2|,|n_3| \leq N}}\widehat{u}_0(n_1) \widehat{u}_0(n_2) \widehat{u}_0(n_3)\right|^2\\
&=\sum_{\substack{|n| \leq 3N \\ j\in \mathbb{Z}}}\left|{\sum_{\substack{(n_1,n_2) \in \mathcal{A}_{n,j} \\ |n_1|, |n_2|, |n-n_1-n_2| \leq N}}\widehat{u}_0(n_1)\widehat{u}_0(n_2)\widehat{u}_0(n-n_1-n_2)}\right|^2,
\end{aligned}
$$
where the last equality uses the substitution $n= n_1+n_2 +n_3$. So
\begin{equation}
\label{sum}\| S_Nu\|_{L^6_{x, t}}^6 
= \sum_{\substack{|n| \leq 3N \\ j\in \mathbb{Z}}}\left|{\sum_{\substack{(n_1,n_2) \in \mathcal{A}_{n,j} \\ |n_1|, |n_2|, |n-n_1-n_2| \leq N}}\widehat{u}_0(n_1)\widehat{u}_0(n_2)\widehat{u}_0(n-n_1-n_2)}\right|^2.
\end{equation}
Note there are $r_{n,j}^{(N)}$ quadruples in $\mathcal{A}_{n,j}$ with $|n_1|, |n_2| \leq N$ by definition.

Thus one can bound the left hand side of (\ref{sum}) using Cauchy-Schwarz as
\begin{equation}\label{sum2}
\begin{aligned}
\sum_{\substack{|n| \leq 3N \\ j\in \mathbb{Z}}}&{\left|{\sum_{\substack{ (n_1,n_2) \in \mathcal{A}_{n,j} \\ |n_1|, |n_2|, |n-n_1-n_2| \leq N}}\widehat{u}_0(n_1)\widehat{u}_0(n_2)\widehat{u}_0(n-n_1-n_2)}\right|^2}\\
&\leq \sum_{\substack{|n| \leq 3N \\ j\in \mathbb{Z}}}\left({r_{n, j}^{(N)}{\sum_{\substack{(n_1,n_2) \in \mathcal{A}_{n,j} \\ |n_1|, |n_2|, |n-n_1-n_2| \leq N}}\left|\widehat{u}_0(n_1)\widehat{u}_0(n_2)\widehat{u}_0(n-n_1-n_2)\right|^2}}\right)\\
&\leq \left(\sup_{n, j \in \mathbb{Z}}{r_{n, j}^{(N)}}\right){\sum_{n, j \in \mathbb{Z}}\sum_{\substack{(n_1,n_2) \in \mathcal{A}_{n,j} \\ |n_1|, |n_2|, |n-n_1-n_2| \leq N}}\left|\widehat{u}_0(n_1)\widehat{u}_0(n_2)\widehat{u}_0(n-n_1-n_2)\right|^2}\\
&= \left(\sup_{n, j \in \mathbb{Z}}{r_{n, j}^{(N)}}\right){{\sum_{ |n_1|, |n_2|, |n_3| \leq N}\left|\widehat{u}_0(n_1)\widehat{u}_0(n_2)\widehat{u}_0(n_3)\right|^2}}.
\end{aligned}
\end{equation}
The last equality of (\ref{sum2}) is seen from the fact that every triple $(n_1, n_2, n)$ gives a unique $j$ such that the quadruple $(n_1, n_2, n, j)$ satisfies the original equation, and substitute back in $n_3 = n-n_1-n_2$.

Finally notice that the expansion of
$$
\left(\sup_{n, j \in \mathbb{Z}}{r_{n, j}^{(N)}}\right) \left(\sum_{|n|\leq N}{|\widehat{u}_0(n)|^2}\right)^3
$$ 
gives all the terms in (\ref{sum2}) as well as additional nonnegative terms due to the coefficients of $3$; therefore
$$
\| S_Nu\|_{L^6_{x,t}}^6 
\leq \left(\sup_{n, j \in \mathbb{Z}}{r_{n,j}^{(N)}}\right)
\left(\sum_{|n|\leq N}{|\widehat{u}_0(n)|^2}\right)^3.
$$
Now take the 6th root of both sides and apply Lemma \ref{diofinal} to bound $$\sup_{n, j \in \mathbb{Z}} r_{n, j}^{(N)} \lesssim N^{\epsilon},$$ which gives the desired
$$
\| S_Nu\|_{L^6_{x,t}} 
\lesssim N^{\epsilon}\left(\sum_{|n|\leq N}{|\widehat{u}_0(n)|^2}\right)^{\frac{1}{2}} = N^{\epsilon}\|S_{N}u_0\|_{L^2_x}.
$$
\end{proof}

\begin{proof}[Proof of theorem \ref{str}]
We shall employ a dyadic decomposition argument. For $j\in\mathbb{N}$, compute
\begin{equation}\label{2adic}
\begin{aligned}
\| S_{2^{j+1}}u - S_{2^{j}}u\|_{L^6_{x, t}} 
\lesssim 2^{(j+1)\epsilon}\|S_{2^{j+1}}u_0 - S_{2^{j}}u_0\|_{L^2_{x}}
\lesssim 2^{j\epsilon}\left|\sum_{2^j \leq |n| < 2^{j+1}}{  |\widehat{u}_0(n)|^2}\right|^{\frac{1}{2}}.
\end{aligned}
\end{equation}
By Minkowski's inequality, one finds for $k\in\mathbb{N}$ that
\begin{equation}\label{decomp}
\begin{aligned}
\|S_{2^k}u\|_{L^6_{x,t}} 
&\leq |\widehat{u}_0(0)| + \sum_{j=1}^{k}{\|S_{2^{j+1}}u - S_{2^{j}}u\|_{L^6_{x,t}}} \\
&\lesssim \sum_{j=0}^k2^{j\epsilon}\left|\sum_{2^j \leq |n| < 2^{j+1}}{  |\widehat{u}_0(n)|^2}\right|^{\frac{1}{2}}
\lesssim\sum_{j=0}^k2^{-j\epsilon}\left|\sum_{2^j \leq |n| < 2^{j+1}}{  |n|^{4\epsilon}|\widehat{u}_0(n)|^2}\right|^{\frac{1}{2}}.
\end{aligned}
\end{equation}
By the Cauchy-Schwarz inequality, the right-hand-side is bounded by
$$
\left(\sum_{j=0}^k 2^{-j\epsilon}\right)^{\frac{1}{2}}
\left(\sum_{|n| < 2^{k+1}}{  |n|^{2\epsilon}|\widehat{u}_0(n)|^2}\right)^{\frac{1}{2}}
\leq C_\epsilon \|u_0\|_{H^{2\epsilon}}.
$$
Taking the limit $k\to \infty$ in (\ref{decomp}), obtain the desired result.
\end{proof}

\begin{proof}[Proof of Corollary \ref{lrhxstr}]
The techniques of this proof is a generalization of the previous arguments. Consider some integer $m\geq 3$. Then one can obtain an estimate for $\|S_Nu\|_{L^{2m}_{x,t}}$ similarly as in Lemma \ref{partsum}. consider $r^{(N),m}_{n, j}$ which is the number of solutions over the integers to 
$$
P(n_1) + P(n_2) + \dots P(n_{m-1}) + P(n- n_1-n_2 - \dots - n_{m-1}) = j
$$ 
with each $|n_i| \leq N$. Then by the same proof as Theorem \ref{diofinal} for each choice of $n_3, n_4 \dots n_{m-1},$ there are at most $C_{\epsilon}N^{\epsilon}$ solutions for $n_1, n_2$ with magnitude less or equal to $N$. Thus
$$
\sup_{n, j \in \mathbb{Z}}{r^{(N),m}_{n, j}} \leq C_{\epsilon, m} N^{s-3 + \epsilon}.
$$
Then by part of the proof of lemma \ref{partsum} see 
$$
\|S_Nu\|_{L^{2m}_{x,t}} \leq C_{\epsilon, m} N^{\frac{m-3 + \epsilon}{2m}} \|S_Nu_0\|_{L^2_{x}}.
$$

Now apply Riesz-Thorin Interpolation Theorem  between the spaces $L^6_{x,t}$ and $L^{2m}_{x,t}$. For $\theta = \frac{(r-6)m}{r(m-3)}$, there holds
$$
\frac{1}{r} = \frac{1-\theta}{6} + \frac{\theta}{2m}.
$$
By interpolation, thus have
$$
\|S_Nu\|_{L^{r}_{x,t}} \leq C_{\epsilon, r} N^{\frac{m-3}{2m} \theta + \epsilon } \|S_Nu_0\|_{L^2_{x}},
$$
or equivalently
$$
\|S_Nu\|_{L^{r}_{x,t}} \leq C_{\epsilon, r} N^{\frac{1}{2} - \frac{3}{r} + \epsilon } \|S_Nu_0\|_{L^2_{x}}.
$$
By the proof of Theorem \ref{str}, one can fold the $N^{\frac{1}{2} - \frac{3}{r} + \epsilon }$ term inside the sum using a $k$-adic decomposition to yield the desired result.
\end{proof}

\section{Linear Inequalities in Bourgain Spaces}\label{3}

In the following two sections, for a function $f$ in the spacetime variable $(x,t)\in\mathbb{T}^2$, shall use $\widehat{f}(n,\tau)$ to denote its space-time Fourier transform.

\subsection{Definitions}
Bourgain in \cite{Bourgain1993} introduced the notion of $X^{s, b}$ and $Y^{s,b}$ spaces (with a different notation) which are suitable norms to the function at hand. It captures the vibrating nature of the solution, i.e. that the space-time Fourier transform of the solution should concentrate near the curve given by the dispersive relation, by introducing a weight in the norm. We also introduce the space $Z^{s,b} = X^{s,b} \cap Y^{s, b- \frac{1}{2}}$. In our case, the $X^{s, b}$ and $Y^{s,b}$ spaces will be given by the norm 
$$
\|u\|_{X^{s,b}} 
:= \left(\sum_{n\in\mathbb{Z}}\langle n\rangle^{2s}\sum_{\tau \in \mathbb{Z}}(1+|\tau - P(n)|)^{2b}|\widehat u(n, \tau)|^2\right)^\frac{1}{2} 
= \|\langle n\rangle ^s \langle \tau - P(n) \rangle ^b \widehat{u}(n, \tau)\|_{l^2_{n, \tau}(\mathbb{Z}^2)},
$$
$$
\|u\|_{Y^{s,b}} 
:= \left[\sum_{n\in\mathbb{Z}}\langle n\rangle^{2s}\left(\sum_{\tau \in \mathbb{Z}}(1+|\tau - P(n)|)^b|\widehat u(n, \tau)|\right)^2\right]^\frac{1}{2} 
= \|\langle n\rangle ^s \langle \tau - P(n) \rangle ^b \widehat{u}(n, \tau)\|_{l^2_{n}l^1_{\tau}},
$$

Recalling the definition \ref{linprop} of the propagator $e^{it\mathcal{L}}$, obtain the following identity:
\begin{equation}\label{bourgainsobolev}
\|u\|_{X^{s,b}} = \|e^{it\mathcal{L}}u\|_{H^{s}_xH^b_t}.
\end{equation}

\begin{lemma}\label{ysbxsbembed}
We have the embedding $X^{s, b+ \delta} \subset Y^{s, b - \frac{1}{2}}$ for $\delta >0$.
\end{lemma}
\begin{proof}
Write out the $Y^{s,b}$ norm and separate the dispersive interaction term into two pieces, and then apply Cauchy-Schwarz on $\tau$ with the second term being the denominator:
$$
\begin{aligned}
\|f\|_{Y^{s,b - \frac{1}{2}}} 
&= \left\|\widehat{f} \langle n \rangle^s \frac{\langle \tau - P(n) \rangle^{b+\delta}}{\langle \tau -P(n) \rangle^{\frac{1}{2} + \delta}} \right\|_{l^2_nl^1_{\tau}}\\
&\leq \|\widehat{f} \langle n \rangle^s \langle \tau - P(n) \rangle^{b+ \delta} \|_{l^2_{n, \tau}} \|\langle \tau - P(n) \rangle^{-(\frac{1}{2} + \delta) } \|_{l^2_{\tau}}.
\end{aligned}
$$
Then the first term is just $\|f\|_{X^{s, b}}$, and one can write the second term as 
$$
\|\langle \tau - P(n) \rangle^{-(\frac{1}{2} + \delta) } \|_{l^2_{\tau}} = \left(\sum_{\tau \in \mathbb{Z}}{\langle \tau \rangle^{-(1 + 2\delta)}}\right)^{\frac{1}{2}} \lesssim_{\delta} 1,
$$
since the middle piece is a well known convergent sum for $\delta >0$. We are left with the result.
\end{proof}

\begin{lemma}\label{3.9embed}
We have the embedding $Y^{s, 0} \subset L_t^{\infty}H^s_x$.
\end{lemma}
The proof is a direct application of the Hausdorff-Young Inequality.

\subsection{Restating the Strichartz Inequalities}
The Strichartz estiamtes in Theorem \ref{str} can be restated in terms of the embedding of Bourgain spaces.
\begin{proposition}[Bourgain's $L^6$ Strichartz Inequality]\label{l6str}
We have the embedding $Y^{\epsilon,0 } \subset L^6_{x,t}$ for positive $\epsilon$.
\end{proposition}

\begin{proposition}\label{lrxsbstr}
For positive $\epsilon$ and $r \geq 6$ have the embedding $Y^{\frac{1}{4} - \frac{3}{2r} + \epsilon, 0} \subset L^r_{x,t}.$
\end{proposition}

Now turn to the proof of main Theorem \ref{L4Ysb}. The argument in our paper follows Chapouto's proof for the KdV equation in \cite{Chapouto2018}. There are differences due to the different dispersive relation for the Dysthe equation.
Before proving this theorem, we need a lemma.

\begin{lemma}\label{AuxYsb}
Let 
$$
f_j = \sum_{\substack{n, \tau \in \mathbb{Z} \\ 2^{j-1} < \langle\tau - P(n)\rangle \leq 2^j}}{\widehat{f}(n, \tau)e^{inx + it\tau}}
$$
so that $f = \sum_{j \in \mathbb{Z}}{f_j}$. For $k\geq 0$, there holds
$$
\|f_jf_{j+k}\|_{L^2_{x,t}} 
\lesssim 2^{\frac{2j}{3} + \frac{k}{6}}\|f_j\|_{L^2_{x,t}}\|f_{j+k}\|_{L^2_{x,t}}.
$$
\end{lemma}

\begin{proof}
By Plancherel's Theorem the left hand side is equal to 
\begin{equation}\label{1temp}
\begin{aligned}
\|\widehat{f_j} *_{n, \tau} \widehat{f_{j+k}}\|_{l^2_{n,\tau}} 
&= \left\|\sum_{n_1+n_2 = n}{\sum_{\tau_1 + \tau_2 = \tau}}{\widehat{f_j}({n_1, \tau_1})\widehat{f_{j+k}}({n_2, \tau_2})}\right\|_{l^2_{n, \tau}}\\
&=\left(\sum_{n}\sum_{\tau}\left|\sum_{n_1+n_2 = n}{\sum_{\tau_1 + \tau_2 = \tau}{\widehat{f_j}({n_1, \tau_1})\widehat{f_{j+k}}({n_2, \tau_2})}}\right|^2 \right)^{\frac{1}{2}}.
\end{aligned}
\end{equation}
Now split the right-hand-side into two parts $S_1$ and $S_2$, where $S_1$ consists of the terms with $|n| \gtrsim 2^{\frac{k+j}{3}}$ and $S_2$ has the remaining terms with $|n| \ll 2^{\frac{k+j}{3}}$.

\textbf{Bounding the sum $S_1$}: We study $S_1$ first, whence $|n| \gtrsim 2^{\frac{k+j}{3}}$. By Cauchy-Schwarz and using the regions of support of $\widehat{f_j}$ and $\widehat{f_{j+k}}$ have 
$$
\begin{aligned}
S_1
& \leq \left(\sup_{n,\tau}M_{n,\tau}^{(j,k)}\right)^{\frac{1}{2}}
\left(\sum_{n}\sum_{\tau}\sum_{n_1+n_2 = n}{\sum_{\tau_1 + \tau_2 = \tau}|{\widehat{f_j}({n_1, \tau_1})\widehat{f_{j+k}}({n_2, \tau_2})|^2}} \right)^{\frac{1}{2}}\\
& = \left(\sup_{n,\tau}M_{n,\tau}^{(j,k)}\right)^{\frac{1}{2}}
\left(\sum_{n_1, n_2}\sum_{\tau_1, \tau_2}|{\widehat{f_j}({n_1, \tau_1})\widehat{f_{j+k}}({n_2, \tau_2})|^2} \right)^{\frac{1}{2}}\\
& = \left(\sup_{n,\tau}M_{n,\tau}^{(j,k)}\right)^{\frac{1}{2}}\|f_j\|_{L^2_{x,t}}\|f_{j+k}\|_{L^2_{x,t}}
\end{aligned}
$$
where we abbreviate
\begin{equation}\label{tempsum}
M_{n,\tau}^{(j,k)}
:={\sum_{n_1 +n_2 = n}\sum_{\substack{\tau_1 + \tau_2 = \tau \\ \tau_1 - P(n_1) \sim 2^j \\ \tau_2 - P(n_2) \sim 2^{j+k} }}{1}}
\end{equation}

Now fix $n, \tau$ and focus on the sum (\ref{tempsum}). For this sum have $\tau_1 = P(n_1)+O(2^j)$ and $\tau_2 = P(n_2)+O(2^{j+k}) $, so $\tau = P(n_1) + P(n_2) + O(2^{j+k})$ since $k\geq 0$. Then the inner sum is determined by the number of values $\tau_1$ can take given $\tau, n, n_1, n_2$ which is at most $O(2^j)$ since $\tau_1 = P(n_1) + O(2^j)$. Thus one can bound
\begin{equation}\label{tempsum2}
M_{n,\tau}^{(j,k)}
\lesssim 2^j \sum_{\substack{n_1 + n_2 = n \\ P(n_1) + P(n_2) = \tau + O(2^{j+k})}}{1}
\end{equation}
for a fixed $n, \tau$. 

next bound the number of integer solutions to the two conditions $n_1 + n_2 = n$ and $P(n_1) + P(n_2) = \tau + O(2^{j+k})$ for $n_1, n_2$. Notice 
$$
P(n) - \tau = P(n) - P(n_1) - P(n_2) + O(2^{j+k}) = 3nn_1n_2 - 4n_1n_2+O(2^{j+k}),
$$
which yields
\begin{equation}\label{P+O}
(3n-4)n_1(n-n_1) = P(n) - \tau + O(2^{j+k}).
\end{equation}
Since $n$ is a nonzero integer and satisfies $|n| \gtrsim 2^{\frac{k+j}{3}} $
$$
\left|\frac{1}{4-3n}\right| \sim \frac{1}{|n|} \lesssim 2^{-\frac{k+j}{3}},
$$
so one can convert (\ref{P+O}) to
$$
n_1(n-n_1)
=\frac{P(n) - \tau}{4-3n}  + O\left(2^{\frac{2(j+k)}{3}}\right)
$$
solving gives
$$
n_1 = \frac{n}{2} \pm \sqrt{\frac{n^2}{4} - \frac{P(n) - \tau}{4-3n} + O\left(2^{\frac{2(j+k)}{3}}\right)}
$$
Then the interval of solutions for each of the plus or minus roots for $n_1$ has length given by 
$$\sqrt{k + O\left(2^{\frac{2(j+k)}{3}}\right)} - \sqrt{k}$$ for some $k$, which is maximized at $k \lesssim 2^{\frac{2(j+k)}{3}}$ and gives length $O\left(2^{\frac{j+k}{3}}\right)$. Since $n,\tau$ are fixed, the integer $n_1$ can take on at most $O(2^{\frac{j+k}{3}})$ values and thus bound $M_{n,\tau}^{(j,k)}\lesssim 2^j2^{\frac{j+k}{3}} = 2^{\frac{4j+k}{3}}$. So
$$
S_1\lesssim 2^{\frac{4j +k}{6}}\|f_j\|_{L^2_{x,t}}\|f_{j+k}\|_{L^2_{x,t}}.
$$

\textbf{Bounding the sum $S_2$}: As for the sum $S_2$, whence $|n| \ll 2^{\frac{k+j}{3}}$, we proceed as follows. By Minkowski's inequality,
$$
\begin{aligned}
S_2
&= \left\|\sum_{n_1}{\sum_{\tau_1}}{\widehat{f_j}({n_1, \tau_1})\widehat{f_{j+k}}({n- n_1, \tau - \tau_1})}\right\|_{l^2_{n, \tau}}\\
& \leq \left\|\sum_{n_1}\sum_{\tau_1}\|\widehat{f_j}(n_1, \tau_1)\widehat{f_{j+k}}(n- n_1, \tau - \tau_1)\|_{l^2_{\tau}} \right\|_{l^2_n}
= \left\|\sum_{n_1}\|\widehat{f_{j+k}}(n- n_1, \tau)\|_{l^2_{\tau}}\sum_{\tau_1}\left|\widehat{f_j}(n_1, \tau_1)\right| \right\|_{l^2_n}.
\end{aligned}
$$
Applying Cauchy-Schwarz inequality to the sum $\sum_{\tau_1}$, one finds
$$
\begin{aligned}
S_2
&\lesssim \left\|\sum_{n_1}\left(\sum_{\tau_1 = P(n_1) + O(2^j)}{1}\right)^{\frac{1}{2}}{\|\widehat{f_j}({n_1, \tau})\|_{l^2_{\tau}}\|\widehat{f_{j+k}}(n- n_1, \tau)\|_{l^2_{\tau}}} \right\|_{l^2_n}\\
& \lesssim 2^{\frac{j}{2}}\left\|\sum_{n_1}{\|\widehat{f_j}({n_1, \tau})\|_{l^2_{\tau}}\|\widehat{f_{j+k}}({n- n_1, \tau})}\|_{l^2_{\tau}} \right\|_{l^2_n}.
\end{aligned}
$$
Now let's expand the $l^2_{n}$ norm. Keeping in mind the restriction $|n| \ll 2^{\frac{k+j}{3}}$, one finds that the last term equals
\begin{equation}\label{2temp}
2^{\frac{j}{2}}\left(\sum_{|n| \ll 2^{\frac{k+j}{3}}}{\left|\sum_{n_1}{\|\widehat{f_j}({n_1, \tau})\|_{l^2_{\tau}}\cdot\|\widehat{f_{j+k}}({n- n_1, \tau})}\|_{l^2_{\tau}}\right|^2}\right)^{\frac{1}{2}}
\end{equation}
So applying Cauchy-Schwarz inequality on the sum $\sum_{n_1}$, and using Parseval Theorem, we obtain
$$
\begin{aligned}
S_2
&\leq 2^{\frac{j}{2}} \left(\sum_{|n| \ll 2^{\frac{k+j}{3}}}\|\widehat{f_j}\|^2_{l^2_{n, \tau}}\|\widehat{f_{j+k}}\|^2_{l^2_{n, \tau}}\right)^{\frac{1}{2}}\\
&= \left(\sum_{|n| \ll 2^{\frac{k+j}{3}}}{1}\right)^{\frac{1}{2}}\|\widehat{f_j}\|_{l^2_{n, \tau}}\|\widehat{f_{j+k}}\|_{l^2_{n, \tau}} \lesssim 2^{\frac{4j+k}{6}}\|f_j\|_{L^2_{x, t}}\|f_{j+k}\|_{L^2_{x, t}}
\end{aligned}
$$ 
as desired.
\end{proof}
\begin{proof}[Proof of theorem \ref{L4Ysb}]
There holds 
$$
\begin{aligned}
\|f\|_{L^4_{x,t}}^2
&= \|f^2\|_{L^2_{x,t}} \\
&= \left\|\sum_{j, j'\in \mathbb{Z}}f_j f_{j'}\right\|_{L^2_{x,t}}
\sim \left\|\sum_{j'\leq j}{f_j f_{j'}}\right\|_{L^2_{x,t}}
\leq \sum_{j\leq j'}{\|f_j f_{j'}\|_{L^2_{x,t}}},
\end{aligned}
$$
with the last inequality following from Minkowski.
By Lemma \ref{AuxYsb} and Cauchy-Schwarz get
$$
\begin{aligned}
\|f\|_{L^4_{x,t}}^2
&\lesssim \sum_{j\in \mathbb{Z}}{\sum_{k\geq 0}{\|f_jf_{j+k}\|_{L^2_{x,t}}}}\\
& \leq \sum_{k\geq 0}{2^{-\frac{k}{6}}} {\sum_{j \in \mathbb{Z}}}{(2^\frac{j}{3}\|f_j\|_{L^2_{x,t}})(2^\frac{j+k}{3}\|f_{j+k}\|_{L^2_{x,t}})}\\
& \leq \sum_{k\geq 0}{2^{-\frac{k}{6}}} \left({\sum_{j \in \mathbb{Z}}}{2^\frac{2j}{3}\|f_j\|^2_{L^2_{x,t}}}\right)^{\frac{1}{2}} \left({\sum_{j \in \mathbb{Z}}}{2^\frac{2(j+k)}{3}\|f_{j+k}\|^2_{L^2_{x,t}}}\right)^{\frac{1}{2}}\\
& = \sum_{k\geq 0}{2^{-\frac{k}{6}}} {\sum_{j \in \mathbb{Z}}}{2^\frac{2j}{3}\|f_j\|^2_{L^2_{x,t}}}\\
& \sim \sum_{j \in \mathbb{Z}}}{2^\frac{2j}{3}\|f_j\|^2_{L^2_{x,t}}
\end{aligned}
$$
with the last one following due to the convergent geometric series. Therefore
$$
\|f\|_{L^4_{x,t}}^2 \lesssim \sum_{j\in \mathbb{Z}}{2^{\frac{2j}{3}}\|f_j\|^2_{L^2_{x,t}}}.
$$
But also have by Parseval' identity,
$$
\begin{aligned}
\sum_{j\in \mathbb{Z}}{2^{\frac{2j}{3}}\|f_j\|^2_{L^2_{x,t}}} 
&= \sum_{j\in \mathbb{Z}}{2^{\frac{2j}{3}}\sum_{n, \tau \in \mathbb{Z}}{|\widehat{f_j}(n, \tau)|^2}}\\
&\leq \sum_{j, n, \tau\in \mathbb{Z}}{\langle\tau -P(n)\rangle^{\frac{2}{3}}|\widehat{f_j}(n, \tau)|^2}\\
&= \sum_{j \in \mathbb{Z}}{\|f_j\|_{X^{0, \frac{1}{3}}}^2}
=\|f\|_{X^{0, \frac{1}{3}}}^2,
\end{aligned}
$$
where the last equality follows from orthogonality since the Fourier bases of the $f_j$s have disjoint regions of support. Thus 
$$
\|f\|_{L^4_{x,t}} \lesssim \|f\|_{X^{0, \frac{1}{3}}}
$$
as desired.
\end{proof}

\subsection{Other Linear Inequalities}

Now present the following three important theorems without proof as they are well-known and do not use the specific $X^{s,b}$ norm for the Dysthe equation. Refer to (among other places) \cite{Chapouto2018} or \cite{CKSTT} for proofs. Recall introduced a smooth bump function $\eta(t)$ that is $1$ on the interval $[-1, 1]$ and supported only on $[-2, 2]$. It is used in the proofs of the following. The bump function is vital to proving local well-posedness as it helps improve our estimates. The following controls the $X^{s,b}$ norm of the linear part:
\begin{proposition}[Homogeneous Linear Estimate]\label{hle}
This inequality helps control the $X^{s,b}$ norm of the linear term from the Duhamel Representation (\ref{duhamelrep}):
$$\|\eta(t)e^{it\mathcal{L}}u_0 \|_{X^{s,b}} \lesssim_{\eta, b} \|u_0\|_{H^s_x}.$$
\end{proposition}

\begin{proposition}[Inhomogeneous Linear Estimate]\label{ile}
The following proposition helps control the $X^{s,b}$ norm of the nonlinear term from the Duhamel Representation:
$$
\left\|\eta(t)\int_{0}^{t}{e^{i(t-t')\mathcal{L}}F(t')dt'}\right\|_{X^{s,b}} \lesssim \|F\|_{X^{s,b-1}}.
$$
\end{proposition}
The next proposition is where the bump function is vital to achieve a contraction mapping argument.

\begin{proposition}[Time Localisation Estimate]\label{tle}
For $-\frac{1}{2} < b' \leq b < \frac{1}{2}$ and a choice of $T<1$,
$$
\left\|\eta\left(\frac{t}{T}\right) u\right\|_{X^{s,b'}} \lesssim_{\eta, b', b} T^{b-b'}\|u\|_{X^{s,b}}.
$$
\end{proposition}

\section{Bilinear and Trilinear Estimates}\label{4}

\subsection{Bilinear Estimates}
Note that the nonlinear terms in the Dysthe equation are at most trilinear in order. To make the trilinear estimates first must make bilinear estimates. use the Strichartz estimates proved in the previous sections to do so. Let $\mathbb{P}$ be the projection onto mean zero in $x$, that is $\mathbb{P}u = u - \int_{\mathbb{T}}{udx}$. The bilinear estimate was originally proved by Kenig, Ponce, and Vega in \cite{KenigPonceVega1996} for a similar Bourgain space. Here adapt the argument given by Chapouto in \cite{Chapouto2018} to the suitable Dysthe Bourgain space to prove Theorem \ref{bilinear}.

\begin{remark} In the original statement of the Theorem \ref{bilinear} of \cite{KenigPonceVega1996}, the left hand side was $\|\mathbb{P}(\mathbb{P}(u_1)\mathbb{P}(u_2))\|_{X^{s, \frac{-1}{2}}}$, but due to the altered dispersive relation with the fact that $3n - 4 \not \in \mathbb{Z}$ for any integer $n$, the extra outer projection is unnecessary.
\end{remark}

\begin{proof}[Proof of Theorem \ref{bilinear}]
In the following let $\sigma_i = \langle \tau_i - P(n_i) \rangle $ for $i = \emptyset, 1, 2$. Since the projection operator $\mathbb{P}$ is already applied, the Fourier modes $n, n_1, n_2$ cannot be zero, thus $|n_i| \sim \langle n_i \rangle$.

First consider just the $X^{s, -\frac{1}{2}}$ norm on the left hand side. By expanding out terms on the left hand side, obtain
$$
\begin{aligned}
LHS &= \left\|\langle n \rangle^s \sigma ^\frac{-1}{2} \widehat{u_1u_2} \right\|_{l^2_{n,\tau}}
\\&= \left\|\sum_{n=n_1+n_2}\sum_{\tau = \tau_1+\tau_2}\frac{{\langle n\rangle}^{s}}{\sigma^\frac{1}{2}} \widehat{u_1}(n_1, \tau_1) \widehat{u_2}(n_2, \tau_2) \right\|_{l^2_{n,\tau}}.
\end{aligned}
$$
Let $\widehat{f_i} = \widehat{u}_i\langle n_i\rangle^{s-1}\sigma_j^{\frac{1}{2}}$. Then $\|\widehat{f_i}\|_{l^2_{n,\tau}} = \|u_i\|_{X^{s-1, \frac{1}{2}}}$
and
 \begin{equation}\label{eqbi1}
 \begin{aligned}
LHS &= \left\|\sum_{n=n_1+n_2}\sum_{\tau = \tau_1+\tau_2}\frac{{\langle n\rangle}^{s-\frac{1}{2}}\langle n\rangle^{\frac{1}{2}}}{(\langle n_1 \rangle \langle n_2 \rangle) ^{s-1})}(\sigma_1\sigma_2\sigma)^\frac{-1}{2} \widehat{f_1}(n_1, \tau_1) \widehat{f_2}(n_2, \tau_2)\right\|_{l^2_{n,\tau}}
\\
& \lesssim \left\|\sum_{n=n_1+n_2}\sum_{\tau = \tau_1+\tau_2}\left(\frac{\langle n\rangle \langle n_1 \rangle \langle n_2 \rangle}{\sigma\sigma_1\sigma_2}\right)^{\frac{1}{2}} \widehat{f_1}(n_1, \tau_1) \widehat{f_2}(n_2, \tau_2)\right\|_{l^2_{n,\tau}}.
\end{aligned}
\end{equation}
Here used Minkowski's inequality, together with with $\langle n_1 + n_2 \rangle \lesssim \langle n_1 \rangle \langle n_2 \rangle $ which applies since $s - \frac{1}{2} \geq 0 $.

Next, note that $P(n) - P(n_1) - P(n_2) = n_1n_2(3n-4)$; since $n$ is an integer, see that $3n-4$ is nonzero and have $\langle n \rangle \leq |3n-4|$. Thus 
$$
\tau-P(n) = \tau_1 -P(n_1) + \tau_2 - P(n_2) -n_1n_2(3n-4),
$$
so by Minkowski's inequality,
$$
|\tau -P(n)| + |\tau_1 - P(n_1)| + |\tau_2 -P(n_2)| \geq |n_1n_2(3n-4)| \geq |n_1n_2| \langle n \rangle.
$$
Therefore,
$$
\sigma + \sigma_1 +\sigma_2 \geq  |\tau -P(n)| + |\tau_1 - P(n_1)| + |\tau_2 -P(n_2)| \geq \langle n \rangle |n_1n_2|.
$$
thus obtain
\begin{equation}\label{eqbi2}
\max(\sigma, \sigma_1, \sigma_2) 
\gtrsim \langle n \rangle \langle n_1 \rangle \langle n_2 \rangle.
\end{equation}

Now break into three different cases (cases two and three are symmetrical so present two cases only). Let $F_i = \frac{f_i}{\sigma_i^\frac{1}{2}}$ in both cases.

\textbf{Case 1:} $\sigma = \max(\sigma, \sigma_1, \sigma_2)$.

Using (\ref{eqbi1}), (\ref{eqbi2}) and our case 1 restriction
$$
\begin{aligned}
LHS 
&\lesssim \left\|\sum_{n=n_1+n_2}\sum_{\tau = \tau_1+\tau_2}{\widehat{F_1}(n_1, \tau_1) \widehat{F_2}(n_2, \tau_2)}\right\|_{l^2_{n,\tau}}\\
&=\|F_1F_2\|_{L^2_{x,t}}
\\ &\leq \|F_1\|_{L^4_{x,t}}\|F_2\|_{L^4_{x,t}}
\\ &\lesssim \|F_1\|_{X^{0, \frac{1}{3}}}\|F_2\|_{X^{0, \frac{1}{3}}}
\\&= \|f_1\|_{X^{0, \frac{-1}{6}}}\|f_2\|_{X^{0, \frac{-1}{6}}}
\\&= \|u_1\|_{X^{s-1, \frac{1}{3}}}\|u_2\|_{X^{s-1, \frac{1}{3}}}.
\end{aligned}
$$
Here we used the $L^4$ Strichartz inequality. This completes case 1. 

\textbf{Case 2}: $\sigma_1 = \max(\sigma, \sigma_1, \sigma_2)$. Without loss of generality this also proves the case where $\sigma_2$ is the greatest.
Then the left hand side becomes
$$
\left\|\sum_{n=n_1+n_2}\sum_{\tau = \tau_1+\tau_2}\frac{\widehat{f_1}(n_1, \tau_1) \widehat{f_2}(n_2, \tau_2)}{\sigma^{\frac{1}{2}}\sigma_2^{\frac{1}{2}}}\right\|_{l^2_{n,\tau}}
$$
which is controlled as, by the $L^4$ Strichartz inequality,
$$\begin{aligned}
LHS &\lesssim \sup_{\|\widehat{f_3}\|_{l^2_{n,\tau}} \leq 1}\sum_{n= n_1 +n_2}\sum_{\tau=\tau_1 + \tau_2}\frac{\widehat{f_3}(n, \tau)\widehat{f_1}(n_1, \tau_1) \widehat{f_2}(n_2, \tau_2)}{\sigma^{\frac{1}{2}}\sigma_2^{\frac{1}{2}}}
\\&=\sup \langle \widehat{f_1F_2}, \widehat{F_3} \rangle = \sup \langle f_1F_2, F_3 \rangle
\\& \leq \sup \|f_1\|_{L^2_{x,t}} \|F_2F_3\|_{L^2_{x,t}} \leq \sup \|f_1\|_{L^2_{x,t}} \|F_2\|_{L^4_{x,t}} \|F_3\|_{L^4_{x,t}}
\lesssim \sup \|\widehat{f_1}\|_{l^2_{n,\tau}} \|F_2\|_{X^{0, \frac{1}{3}}} \|F_3\|_{X^{0, \frac{1}{3}}}
\\& = \sup \|u_1\|_{X^{s-1, \frac{1}{2}}} \|u_2\|_{X^{s-1, \frac{1}{3}}} \|f_3\|_{X^{0, \frac{-1}{6}}} \lesssim \sup \|u_1\|_{X^{s-1, \frac{1}{2}}} \|u_2\|_{X^{s-1, \frac{1}{3}}} \|f_3\|_{X^{0, 0}}
\\& \leq \|u_1\|_{X^{s-1, \frac{1}{2}}} \|u_2\|_{X^{s-1, \frac{1}{3}}}.
\end{aligned}$$
This completes the second case.

Thus
$$
\begin{aligned}
LHS &\lesssim \|u_1\|_{X^{s-1, \frac{1}{3}}} \|u_2\|_{X^{s-1, \frac{1}{3}}} + \|u_1\|_{X^{s-1, \frac{1}{2}}} \|u_2\|_{X^{s-1, \frac{1}{3}}} + \|u_1\|_{X^{s-1, \frac{1}{3}}} \|u_2\|_{X^{s-1, \frac{1}{2}}} 
\\ &\lesssim \|u_1\|_{X^{s-1, \frac{1}{2}}} \|u_2\|_{X^{s-1, \frac{1}{3}}} + \|u_1\|_{X^{s-1, \frac{1}{3}}} \|u_2\|_{X^{s-1, \frac{1}{2}}}
\end{aligned}
$$
completing the proof for the $X^{s, -\frac{1}{2}}$ norm.

Now it remains to bound the norm 
$$
\|u_1u_2\|_{Y^{s, -1}} 
= \left\|\frac{\langle n \rangle^s \widehat{u_1u_2}}{\langle\tau - P(n)\rangle} \right\|_{l^2_nl^1_{\tau}} 
= \left\|\sum_{n_1+n_2 =n}\sum_{\tau_1+\tau_2 = \tau}\frac{\langle n \rangle^s}{\sigma} \widehat{u_1}(n_1, \tau_1)\widehat{u_2}(n_2, \tau_2) \right\|_{l^2_nl^1_{\tau}}.
$$

This is done by cases of $\langle n \rangle \lesssim \sigma $, where we bound by the $Y^{s-1, 0}$ norm or $\langle n \rangle \gtrsim \sigma$.

\textbf{Case 1:} $\langle n \rangle^{\frac{5}{4}} \lesssim \sigma.$ Then $$\frac{\langle n \rangle^s}{\sigma} \lesssim \langle n \rangle ^{s-\frac{5}{4}} \lesssim \langle n_1 \rangle^{s-\frac{5}{4}} \langle n_2 \rangle ^{s-\frac{5}{4}}$$ where $n_1 + n_2 =n$ which allows us to bound with the $Y^{s-1, 0}$ norm. Let $\mathcal{F}_x f_i = \langle n_i \rangle^{s- \frac{5}{4}} \mathcal{F}_x u_i$. Then using Parseval's, Holder's, and a Sobolev all in $x$, get
$$
\begin{aligned}
&\left\|\sum_{n_1+n_2 =n}\sum_{\tau_1+\tau_2 = \tau}\frac{\langle n \rangle^s}{\sigma} \widehat{u_1}(n_1, \tau_1)\widehat{u_2}(n_2, \tau_2) \right\|_{l^2_nl^1_{\tau}} \\&\lesssim \sum_{\tau_1, \tau_2}\|\sum_{n_1+n_2 =n}\langle n_1 \rangle ^{s-\frac{5}{4}} \widehat{u_1}(n_1, \tau_1) \langle n_2 \rangle ^{s-\frac{5}{4}}\widehat{u_2}(n_2, \tau_2) \|_{l^2_n} = \sum_{\tau_1, \tau_2}\|\mathcal{F}_x(f_1f_2) \|_{l^2_n}
\\& = \sum_{\tau_1, \tau_2}\|f_1f_2\|_{L^2_x}
\\& \lesssim \sum_{\tau_1, \tau_2}\|f_1\|_{L^4_x} \|f_2\|_{L^4_x} = \|f_1\|_{L^4_xl^1_{\tau_1}} \|f_2\|_{L^4_x l^1_{\tau_2}}
\\& \lesssim \|f_1\|_{H^{\frac{1}{4}}_xl^1_{\tau_1}}\|f_2\|_{H^{\frac{1}{4}}_xl^1_{\tau_2}}
\\& = \|\langle n_1 \rangle^{s-\frac{5}{4}} \langle n_1 \rangle^{\frac{1}{4}} \widehat{u_1} \|_{l^2_{n_1}l^1_{\tau_1}} \|\langle n_2 \rangle^{s-\frac{5}{4}} \langle n_2 \rangle^{\frac{1}{4}} \widehat{u_2} \|_{l^2_{n_2}l^1_{\tau_2}} = \|u_1\|_{Y^{s-1, 0}} \|u_2\|_{Y^{s-1, 0}}
\end{aligned}
$$
which is bounded by the product of the desired $Y^{s-1, 0}$ norms.

\textbf{Case 2:} $\langle n \rangle ^{\frac{5}{4}} \gtrsim \sigma.$ 
We treat this case like we did the $X^{s, b}$ norm. Consider now the subcase $\sigma = \max(\sigma, \sigma_1, \sigma_2), $ then by assumption, (\ref{eqbi2}), and Minkowski's
$$
\langle n \rangle^{\frac{5}{4}} \gtrsim \sigma 
\gtrsim \langle n \rangle \langle n_1 \rangle \langle n_2 \rangle \gtrsim \langle n \rangle^2,
$$
which means
$$
\sigma \lesssim \langle n \rangle^{\frac{5}{4}} \lesssim 1.
$$
This is then a trivial subcase, and without loss of generality only need to examine the subcase $\sigma_1 = \max(\sigma, \sigma_1, \sigma_2)$. Applying Lemma \ref{ysbxsbembed} and using the same process as case 2 above, one can control our norm by
$$
\left\|\sum_{n=n_1+n_2}\sum_{\tau = \tau_1+\tau_2}\frac{\widehat{f_1}(n_1, \tau_1) \widehat{f_2}(n_2, \tau_2)}{\sigma^{\frac{1}{2} + \delta}\sigma_2^{\frac{1}{2}}}\right\|_{l^2_{n,\tau}}.
$$

But this extra $\sigma^{\delta}$ term in the denominator does not affect the proof from case 2 above at all, since if we let $G_3 = f_3 \sigma^{-(\frac{1}{2} + \delta)}$ the corresponding term to be bounded is 
$$
\|G_3\|_{L^4_{x,t}} \lesssim \|G_3\|_{X^{0, \frac{1}{3}}} \lesssim \|f_3\|_{X^{0, -\frac{1}{6} + \delta}} \lesssim \|f_3\|_{X^{0, 0}}
$$
as required in the proof for case 2 above, completing our final subcase.
\end{proof}

The next  theorem helps to simplify the nonlinear term with no derivative:
\begin{theorem} \label{bilinear2}
$$\|u_1\mathbb{P}(u_2)\|_{Z^{s, \frac{-1}{2}}} \lesssim \|u_1\|_{Z^{s-1, \frac{1}{2}}}\|u_2\|_{Z^{s-1, \frac{1}{3}}} + \|u_1\|_{Z^{s-1, \frac{1}{3}}}\|u_2\|_{Z^{s-1, \frac{1}{2}}} + \|u_1\|_{Z^{s-1, \frac{1}{2}}} \|u_2\|_{X^{s, 0}}$$
and
$$\|u_1u_2\|_{Z^{s, \frac{-1}{2}}} \lesssim \|u_1\|_{Z^{s-1, \frac{1}{2}}}\|u_2\|_{Z^{s-1, \frac{1}{3}}} + \|u_1\|_{Z^{s-1, \frac{1}{3}}}\|u_2\|_{Z^{s-1, \frac{1}{2}}} + \|u_1\|_{Z^{s-1, \frac{1}{2}}} \|u_2\|_{X^{s, 0}} + \|u_1\|_{X^{s, 0}} \|u_2\|_{Z^{s-1, \frac{1}{2}}}.$$
\end{theorem}
\begin{remark}
It turns out we need the $\|u\|_{X^{s,0}}$ term because of the case where $n_1 =0$, $\tau - P(n) = o(1)$, $\tau_2 - P(n) = \tau_2 - P(n_2) = o(1)$ and $\tau = o(1)$. Then if the terms of $u_2$ are concentrated at $\widehat{u_2}(n_2, \tau_2)$ with $\tau_2 - P(n_2) \ll n_2$ the $\langle n_2 \rangle^s$ cannot be cancelled by the $\sigma^\frac{-1}{2}$. This term will pose a significant obstacle to proving local well-posedness in the next section.
\end{remark}

\begin{proof}

Use the same notation as Theorem \ref{bilinear}. Also let $\widehat{g_2} = \langle n_2 \rangle^s \widehat{u_2}$ and $\widehat{h_1} = \langle n_1 \rangle ^{s-1} \widehat{u_1}$.
By Lemma \ref{ysbxsbembed} $$LHS \lesssim \|\langle n \rangle^s \sigma ^{\frac{-1}{2} + \delta} \widehat{u_1u_2} \|_{l^2_{n,\tau}}.
$$
By the trivial embedding and Theorem \ref{bilinear} then we only need to consider the case $n_1 =0$. Denote this by $I$.

In this case use the facts that $1 \lesssim \langle \sigma \rangle $, $\langle n_1 \rangle = 1$ and $\langle n \rangle = \langle n_2 \rangle $  to get 
$$
I = \left\| \sum_{\tau_1 + \tau_2 = \tau}{\widehat{h_1}(0, \tau_1) \widehat{g_2}(n, \tau_2)}\right\|_{l^2_{n, \tau}}.
$$
Using Parseval's equality with respect to time, then pulling the $\widehat{h_1}$ out of the $l^2_n$ norm, and then again using Parseval's equality with respect to space on $g_2$ we get
$$
I = \left\|\mathcal{F}_x h_1(0, t)  \|g_2 \|_{L^2_x} \right\|_{L^2_t}.
$$
Then by H\"{o}lder's and using Parseval's equality again on $g_2$,
$$
I \lesssim \|\mathcal{F}_x h_1(0, t)\|_{L_t^{\infty}} \|g_2\|_{L^2_{x,t}} = \|\mathcal{F}_x h_1(0, t)\|_{L_t^{\infty}} \|u_2\|_{X^{s, 0}}.
$$
Then use the explicit formula for $\widehat{h_1}(0)$, trivial embedding on the torus and Lemma \ref{3.9embed}
$$
\begin{aligned}
I &= \left\|\int_{\mathbb{T}}{h_1(x,t)dx} \right\|_{L^{\infty}_t} \|u_2\|_{X^{s, 0}}
\\ &\lesssim \|h_1\|_{L^1_xL^{\infty}_t}\|u_2\|_{X^{s, 0}}
\\ &\lesssim \|h_1\|_{L^2_xL^{\infty}_t}\|u_2\|_{X^{s, 0}}
\\& = \|u_1\|_{H^{s-1}_xL^{\infty}_t}\|u_2\|_{X^{s, 0}}
\\& \lesssim \|u_1\|_{Z^{s-1, \frac{1}{2}}} \|u_2\|_{X^{s, 0}},
\end{aligned}
$$
which when combined with the results of Theorem \ref{bilinear} yield the result.
\end{proof}

\begin{theorem} \label{Xs0bilinear}
For $s\geq 0$
$$\|u_1u_2\|_{X^{s, 0}} \lesssim \|u_1\|_{X^{s, \frac{1}{3}}}\|u_2\|_{X^{s, \frac{1}{3}}}.$$
\end{theorem}
\begin{proof}
Let $\widehat{f_i}$ denote $\langle n_i \rangle ^s \widehat{u_i}$ throughout this proof. We first expand and then plug in $\langle n \rangle \lesssim \langle n_1 \rangle \langle n_2 \rangle$ where $n_1 + n_2=n$ by Minkowski's, and apply the $L^4$ Strichartz inequality \ref{L4Ysb}:
$$\begin{aligned}
\|u_1u_2\|_{X^{s, 0}} 
& = \left\|\langle n \rangle ^s \sum_{n_1 +n_2 = n} \sum_{\tau_1 + \tau_2 = \tau}\widehat{u_1}\widehat{u_2}\right\|_{l^2_{n, \tau}}
\\
& \lesssim \left\|\sum_{n_1 +n_2 = n} \sum_{\tau_1 + \tau_2 = \tau}\langle n_1 \rangle^s \widehat{u_1} \langle n_2 \rangle ^s \widehat{u_2}\right\|_{l^2_{n, \tau}}
\\& = \left\|\langle n \rangle ^s \sum_{n_1 +n_2 = n} \sum_{\tau_1 + \tau_2 = \tau}\widehat{f_1}\widehat{f_2}\right\|_{l^2_{n, \tau}}
\\
& = \|f_1f_2\|_{L^2_{x,t}}\\
& \lesssim \|f_1\|_{L^4_{x,t}} \|f_2\|_{L^4_{x,t}}
\\& \lesssim \|f_1\|_{X^{0,\frac{1}{3}}} \|f_2\|_{X^{0,\frac{1}{3}}}
\\& = \|u_1\|_{X^{s, \frac{1}{3}}}\|u_2\|_{X^{s, \frac{1}{3}}}.
\end{aligned}$$
\end{proof}

\begin{theorem}\label{CKSTTbilinear}
For $s\geq \frac{1}{2}$
$$
B=\|u_1u_2\|_{Z^{s-1, \frac{1}{2}}} \lesssim \|u_1\|_{Z^{s, \frac{1}{2} }}\|u_2\|_{Z^{s, \frac{1}{2}}},
$$
where labelled the left hand side $B$ for ease of notation.
\end{theorem}
\begin{proof}
We expand into 
$$
B= \left\|\sum_{n=n_1+n_2} \langle n \rangle ^{s-1} \sigma^\frac{1}{2} \widehat{u_1}(n_1, \tau_1)\widehat{u_2}\right\|_{l^2_{n, \tau}}.
$$
Notice $\tau - P(n) = \tau_1 - P(n_1) + \tau_2 - P(n_2) - n_1n_2(3n-4).$
Then divide into two cases. 

\textbf{Case 1:} $|\tau - P(n)| \lesssim |\tau_1 - P(n_1)|$.\\ Without loss of generality this also encompasses if $|\tau-P(n)|$ less or similar to $\sigma_2$. Then $\sigma \lesssim \sigma_1$ means
$$
B \lesssim \left\|\sum_{n=n_1+n_2} \langle n \rangle ^{s-1} \sigma_1^\frac{1}{2} \widehat{u_1}(n_1, \tau_1)\widehat{u_2}\right\|_{l^2_{n, \tau}}.
$$
Let $F_1 = \sigma_1^\frac{1}{2}\widehat{u_1}$
It remains to bound
$$
\|F_1u_2\|_{X^{s-1, 0}} = \|F_1u_2\|_{H^{s-1}_xL^2_t} \lesssim \left\|\|F_1\|_{L^2_{t}}\|u_2\|_{L^{\infty}_t}\right\|_{H^{s-1}_x}
$$
which by linear estimates on pg.11 of \cite{CKSTT} gives
$$
B \lesssim \|F_1\|_{L^2_tL^{2+}_x} \|u_2\|_{L^{\infty}_tL^{2+}_x}
\lesssim \|F_1\|_{H^{\epsilon}_xL^{2}_t}\|u_2\|_{L^{\infty}_tH^{\epsilon}_x}.
$$
Then applying Lemma \ref{3.9embed}
$$
B \lesssim \|u_1\|_{X^{\epsilon, \frac{1}{2}}}\|u_2\|_{Y^{\epsilon, 0}}.
$$

\textbf{Case 2:} $|\tau-P(n)| \gg |\tau_i - P(n_i)|$ for $i=1, 2$.

Then here must have $|\tau-P(n)| \sim |n_1n_2(3n-4)|$. 
But notice since $n_1, n_2, 3n-4, \tau - P(n)$ cannot be zero so we have 
$$\begin{aligned}
\langle \tau-P(n) \rangle &\sim |n_1n_2(3n-4)|
\\&\sim \langle n_1\rangle \langle n_2 \rangle |n| 
\\&\leq  \langle n_1\rangle \langle n_2 \rangle (|n_1| + |n_2|)
\lesssim \langle n_1\rangle \langle n_2 \rangle (\langle n \rangle ).
\end{aligned} 
$$
This gives
$$
B \lesssim \left\|\sum_{n=n_1+n_2} \langle n \rangle ^{s-\frac{1}{2}} \langle n_1 \rangle ^{\frac{1}{2}}\langle n_2 \rangle ^{\frac{1}{2}}  \widehat{u_1}(n_1, \tau_1)\widehat{u_2}\right\|_{l^2_{n, \tau}}.
$$
But $n=n_1+n_2 \implies \langle n\rangle \lesssim  \langle n_1 \rangle + \langle n_2 \rangle \lesssim \langle n_1 \rangle \langle n_2 \rangle$
since each term must be greater than $1$. Also $s-\frac{1}{2}>0,$ so one can plug this into $B$ to get

$$
B \lesssim \left\|\sum_{n=n_1+n_2} \langle n_1 \rangle^s \widehat{u_1}(n_1, \tau_1)\langle n_2 \rangle ^s\widehat{u_2}\right\|_{l^2_{n, \tau}}.
$$
Let $\widehat{f_i} = \langle n_1 \rangle ^s \widehat{u_i}$. Then applying Parseval's equality, H\"{o}lder's inequality and the $L^4$ Strichartz estimate \ref{L4Ysb}, obtain
$$
\begin{aligned}
B &\lesssim \left\|\sum_{n=n_1+n_2} \widehat{f_1}(n_1, \tau_1)\widehat{f_2}\right\|_{l^2_{n, \tau}}
\\&\lesssim  \|f_1f_2\|_{L^2_{x, t}} \lesssim \|f_1\|_{L^4_{x, t}} \|f_2\|_{L^4_{x, t}}
\\& \lesssim \|u_1\|_{X^{s, \frac{1}{3}}}\|u_2\|_{X^{s, \frac{1}{3}}},
\end{aligned}
$$
finishing this case.

In summary, 
$$
B \lesssim \|u_1\|_{X^{\epsilon, \frac{1}{2}}}\|u_2\|_{X^{\epsilon, \frac{1}{2}}} + \|u_1\|_{X^{s, \frac{1}{3}}}\|u_2\|_{X^{s, \frac{1}{3}}}
$$
which gives the result upon applying the trivial embedding in Bourgain spaces.\\

Next bound the $Y^{s-1, 0}$ norm. Notice one can separate out the $\tau_1, \tau_2$ terms
$$
\|u_1u_2\|_{Y^{s-1, 0}} 
= \left\|\sum_{\tau_1 + \tau_2 = \tau}\sum_{n_1 +n_2 = n} \langle n \rangle^{s-1}\widehat{u_1}(n_1, \tau_1)\widehat{u_2}(n_2, \tau_2)\right\|_{l^2_nl^1_\tau} = \sum_{\tau_1}\sum_{\tau_2}\|\langle n \rangle^{s-1} \widehat{u_1}\widehat{u_2}\|_{l^2_n} 
$$
Now let $\mathcal{F}_xh_i (n_i) = \sum_{\tau_i}{\langle n_i \rangle ^{s-1}\widehat{u_i}(n_i, \tau_i)}$ By successive applications of Minkowski's, Parseval's, and Cauchy Schwarz bound the prior term:
$$
\begin{aligned}
&\lesssim \left\|\sum_{n_1 +n_2 = n} \sum_{\tau_1}\langle n_1 \rangle^{s-1} \widehat{u_1}\sum_{\tau_2}\langle n_2 \rangle^{s-1}\widehat{u_2}\right\|_{l^2_n}\\
& = \|\mathcal{F}_x(h_1h_2)\|_{l^2_n}\\
& = \|h_1h_2\|_{L^2_x} \lesssim \|h_1\|_{L^4_x} \|h_2\|_{L^4_x}
\end{aligned}
$$
By the Sobolev inequality one can bound the prior term by
$$
\|h_1\|_{H^{\frac{1}{4}}_x}\|h_2\|_{H^{\frac{1}{4}}_x} = \sum_{\tau_1}\|\langle n_1 \rangle^{\frac{1}{4}} \langle n_1 \rangle ^{s-1}\widehat{u_1}\|_{l^2_{n_1}}\sum_{\tau_2}\|\langle n_2 \rangle^{\frac{1}{4}} \langle n_2 \rangle ^{s-1}\widehat{u_2}\|_{l^2_{n_2}} = \|u_1\|_{Y^{s- \frac{3}{4}, 0}}\|u_2\|_{Y^{s- \frac{3}{4}, 0}}
$$
which by trivial embeddings gives the desired result.
\end{proof}

Now show a sample trilinear estimate which can be used to control a nonlinear term in the Dysthe equation:
\begin{lemma}\label{trilinear1}
Let $T<1$. Then
$$
\left\|\eta\left(\frac{t}{T}\right)|u|^2u\right\|_{Z^{s, \frac{-1}{2}}} \lesssim T^{\frac{1}{6}}\|u\|^3_{Z^{s, \frac{1}{2}}}.$$
\end{lemma}
\begin{proof}
Denote the left-hand-side by $I$. Notice that the $\eta$ does not affect any manipulations until apply the time localisation estimate, so leave it implied throughout this proof.

Split $|u|^2u = u^2 \cdot u^*.$ Then by Theorem \ref{bilinear2} write 
$$
I \lesssim \|u^2\|_{X^{s-1, \frac{1}{2}}}\|u^*\|_{X^{s-1, \frac{1}{3}}} + \|u^2\|_{X^{s-1, \frac{1}{3}}}\|u^*\|_{X^{s-1, \frac{1}{2}}} + \|u^2\|_{X^{s-1, \frac{1}{2}}} \|u^*\|_{X^{s, 0}} + \|u^2\|_{X^{s, 0}}\|u^*\|_{X^{s-1, \frac{1}{2}}}
$$
Now notice $\|u^*\| = \|u\|$. Then apply Theorem \ref{Xs0bilinear} to the last term and then time localisation estimate to every term with $b= \frac{1}{3}$ or $b=0$ and use the fact that $T<1$ to drop the powers of $T$ higher than $\frac{1}{6}$ and get
$$
I \lesssim T^{\frac{1}{6}}\left(\|u^2\|_{X^{s-1, \frac{1}{2}}}\|u\|_{X^{s-1, \frac{1}{2}}} + \|u^2\|_{X^{s-1, \frac{1}{2}}}\|u\|_{X^{s-1, \frac{1}{2}}} + \|u^2\|_{X^{s-1, \frac{1}{2}}} \|u\|_{X^{s, \frac{1}{2}}} + \|u\|_{X^{s, \frac{1}{2}}}^2\|u\|_{X^{s-1, \frac{1}{2}}}\right)
$$
Finally use the trivial embedding $X^{s, \frac{1}{2}} \subset X^{s-1, \frac{1}{2}}$ and Theorem \ref{CKSTTbilinear} on the first three terms to get

$$
I \lesssim T^{\frac{1}{6}}\|u\|_{X^{s, \frac{1}{2}}}^3
$$
as claimed.
\end{proof}

\section{Obstacles to Well-posedness problem} \label{5}
\subsection{Failure of the Harmonic Analysis Approach}
We try a contraction mapping argument in the $Z^{\frac{1}{2}, \frac{1}{2}}$ space.  By the Duhamel representation (\ref{duhamelrep}) one can define the solution map as
$$
\Gamma_{u_0}(u) = \eta(t) e^{it\mathcal{L}}u_0 + i \eta\left(\frac{t}{T}\right)\int_{0}^{t}{e^{i(t-t')\mathcal{L}}\mathcal{N}(u(t'))dt'}
$$
One difficulty in proving the well-posedness of the Dysthe equation is that the mean is not conserved. Thus take out a variable mean, $a(t),$ but this causes issues later down the line.

Let $a(t)=\int_{\mathbb{T}}{udx}$ and $v=u-a$. We compute $$
a'(t) = \int_{\mathbb{T}}{\frac{i}{2}|u|^2u + \frac{3}{2}|u|^2 \partial_x u + \frac{1}{4}u^2\partial_x{u^*} - \frac{1}{2}iu|\partial_x||u|^2 dx}.
$$
Then our new Duhamel map becomes 
$$
\Gamma_{v_0}(v) = \eta\left(\frac{t}{T}\right) e^{it\mathcal{L}}v_0 + i \eta\left(\frac{t}{T}\right)\int_{0}^{t}{e^{i(t-t')\mathcal{L}}(\mathcal{N}(u(x, t')) + a'(t'))dt'}.
$$
First attempt to show that, for small data $u_0$ under the $Z^{\frac{1}{2}, \frac{1}{2}}$ norm, the mapping $\Gamma_{v_0}(v)$ maps a ball to itself. To illustrate the main obstacles, assume $\|v\|_{Z^{\frac{1}{2}, \frac{1}{2}}} \gg \|a\|_{Z^{\frac{1}{2}, \frac{1}{2}}}$ in what follows. Again also keep the $\eta$ implied since it does not change any of the estimates except the Time Localisation Estimate.

By Minkowski's one can bound the norm of $a'$ in terms of the corresponding norm of the nonlinear term, and then by the Homogenous Linear Estimate \ref{hle} and the Inhomogenous Linear Estimate \ref{ile} and Minkowski's
have 
$$
\|\Gamma_{v_0}(v)\|_{Z^{\frac{1}{2}, \frac{1}{2}}} \lesssim \||v+a|^2(v+a) \|_{Z^{\frac{1}{2}, \frac{-1}{2}}} + \|(v+a)^2\partial_x v^*\|_{Z^{\frac{1}{2}, \frac{-1}{2}}} + \||v+a|^2\partial_x v\|_{Z^{\frac{1}{2}, \frac{-1}{2}}} + \||\partial_x| |v+a|^2 \|_{Z^{\frac{1}{2}, \frac{-1}{2}}}.
$$

Applying trilinear estimate Lemma \ref{trilinear1} and again Minkowski's on the first term, and using the fact that $av = \mathbb{P}(av)$ one can treat the local terms with a single $a$ with the bilinear estimate (\ref{bilinear}), the Time Localisation Estimate (\ref{tle}) and then the last bilinear estimate (\ref{CKSTTbilinear}). The nonlocal term in $v$ only can be treated with Theorem \ref{bilinear} and Theorem \ref{tle} since the $Z^{s, b}$ norm of the Fourier multiplier $|\partial_x|$ is the same as that of $\partial_x$. In total, this gives
$$
\|\Gamma_{v_0}(v)\|_{Z^{\frac{1}{2}, \frac{1}{2}}} \lesssim T^\frac{1}{6}\|v\|_{Z^{\frac{1}{2}, \frac{1}{2}}}^3 + \||a|^2 \partial_x v\|_{Z^{\frac{1}{2}, \frac{-1}{2}}} + \|a^2 \partial_x v^*\|_{Z^{\frac{1}{2}, \frac{-1}{2}}} + \|a^*|\partial_x| v\|_{Z^{\frac{1}{2}, \frac{-1}{2}}} + \|a|\partial_x| v^*\|_{Z^{\frac{1}{2}, \frac{-1}{2}}}.
$$

These remaining terms pose more of a challenge. This is because the projections of the $a$ terms are nonzero, so one cannot apply the bilinear estimate \ref{bilinear}. But applying the bilinear estimate \ref{bilinear2} will give us a term like $$\|\partial_x v\|_{Z^{\frac{1}{2}, b}} = \|v\|_{Z^{\frac{3}{2}, b}},$$ which cannot be controlled by the ${Z^{\frac{1}{2}, \frac{1}{2}}}$ space.

However, this is a good start that leaves us with terms that are purely linear in $v$, so one can conjecture that the space ${Z^{\frac{1}{2}, \frac{-1}{2}}}$ will give low-regularity local-well posedness.

\subsection{Failure of Viscosity and Energy Method}
The viscous version of the Dysthe equation is given by 
$$
\partial_tu_\mu-\mu\partial_x^2 u_\mu+ 8\partial_x u_\mu + 2i\partial_x^2 u_\mu -\partial_x^3 u_\mu = \mathcal{N}(u_\mu).
$$
It is expected that when $\mu\to0$, the solutions $u_\mu$ should converge in some sense to a solution of the Dysthe equation -- provided that one can obtain some a priori estimates for the $u_\mu$'s. The viscos Dysthe equation can be solved by considering the Fourier multiplier $$
e^{tQ_{\mu}}f:=e^{t\mu\partial_x^2 -tP(\partial_x)}f,
$$
and solving the integral equation
$$
u_{\mu}(t)=e^{tQ_\mu}u_{\mu, 0}+\int_0^te^{(t-t')Q_\mu}\mathcal{N}(u_{\mu}(t'))dt'.
$$

Throughout these next steps, we work in the space $H^2(\mathbb{T})$ for simplicity.

Notice that $e^{t\mu\partial_x^2}$ works to smooth out our function over time, thus similarly to the KdV equation as in \cite{IorioIorio2001} one can establish a well-posedness estimate in $H^2$, over a time interval $\sim[0,\mu\|u_0\|_{H^2}^{-1}]$. In order to make this lifespan estimate independent of $\mu$, need to apply the extension principle given in \cite{IorioIorio2001}, which relies on estimates of the form
$$
\partial_t \|u_{\mu}\|_{H^2} \lesssim \Phi(\|u_{\mu}\|_{H^2}),
$$
where $\Phi$ is an increasing function. If $u_\mu$ satisfies the viscous KdV equation, have $\partial_t\|u_{\mu}\| \lesssim \|u\|_{H^2}^2$. This is due to the following cancellation property enjoyed by the nonlinearity of the KdV type equations:
$$
\langle \partial_x^2(u\partial_xu),\partial_x^2u\rangle
\lesssim\|u\|_{H^2}^3.
$$
However, we shall provide a counter example showing that such estimate cannot hold for the viscous Dysthe equation.

So we consider $\partial_t\|u_\mu\|_{H^2_x}^2$. Substituting in the equation satisfied by $u_\mu$, expanding out the right hand side as an inner product, and using integration by parts over the torus, one can bound every term using $\|u_\mu\|_{H^2}$ except
$$
I(u_\mu) := \text{Re}\left\langle i u_\mu \cdot\partial_x^2 |\partial_x||u_\mu|^2  , \partial_x^2 u_\mu \right\rangle .
$$
We aim to show that $I(u)$ cannot be bounded in terms of $\|u\|_{H^2}$. Let $f:\mathbb{N} \to \mathbb{N}$ be an increasing integer-valued function that grows faster than any polynomial. The sequence 
$$
v_n = -i + e^{inx} + \frac{1}{f(n)^2}\left[e^{if(n)x} + e^{i(n-f(n))x}\right]
$$
will be our counter example. We claim that the $H^2$ norm of $v_n$ is on the order of $|n|^2$, but $I(v_n)$ and thus the time derivative of the Sobolev norm of $v_n$ is on the order of $f(n)$. If $\Phi$ is already specified, then one can choose the $f$ to grow so fast that $I(v_n)$ cannot be controlled in terms of $\Phi(\|v_n\|_{H^2})$. In fact, a direct calculation gives
$$
I(v_n) = 4f(n) - 7n + \frac{10n^2}{f(n)} - \frac{10n^3}{f(n)^2} + \frac{5n^4}{f(n)^3} - \frac{n^5}{f(n)^4}
\simeq f(n).
$$
If $\Phi$ is a specified increasing function, then one can just choose $f(n)=2^n\Phi(|n|^2)$, so that 
$$
I(v_n)\gg \Phi(|n|^2)\simeq\Phi(\|v_n\|_{H^2}).
$$
This suggests that any energy methods in $H^s$ space should fail for the 1D periodic Dysthe equation. By some more delicate construction, it is expectable to obtain an ill-posedness result similar as in \cite{GrandeKurianskiStaffilani2020}.\\

\section{Illposedness}\label{6}

We follow a similar steps to \cite{GrandeKurianskiStaffilani2020} section 6, which proved similar ill-posedness results for the 2D Dysthe equation in the real space. 
Consider a given function $u_0$. Then we compute the third iteration under a Picard iteration scheme on the Duhamel mapping of the Dysthe equation for initial value $u_0$
$$u_3 = \int_{0}^{t} {e^{i(t-\tau)\mathcal{L}} \mathcal{N}(e^{i\tau \mathcal{L}} u_0) }.$$
Then for the Dysthe equation to be well-posed, we need the continuous mapping condition from $u_0$ to $u_3$, that is,
$$
\|u_0\|_{H^s} \gtrsim \|u_3\|_{L^{\infty}_tH^s_x}.
$$

Now we turn our attention to Theorem \ref{ill}.
\begin{proof}[Proof of Theorem \ref{ill}]
For a given integer $m$, choose our sample initial function
$$
u_0 = \frac{1}{m^s}(e^{-imx} + e^{-i(m-1)x} + e^{i(m+1)x}).
$$
Clearly $\|u_0\|_{H^s} \sim 1$.

Then we compute the contributions to $\mathcal{F}_x u_3 (n)$ for each nonlinear term in $\mathcal{N}$: $\frac{3}{2}|u|^2\partial_x u$; $\frac{i}{2} u |\partial_x| |u|^2$; and $\frac{1}{4}u^2\partial_x u^*$. We disregard the nonlinearity $|u|^2u$ since it does not have a significant contribution as shown in \cite{GrandeKurianskiStaffilani2020}.

For example, the second nonlinearity's contribution to $\mathcal{F}_x u_3(n)$ can be written as 
$$
\frac{i}{2}e^{-itP(n)}\sum_{n_1, n_2 \in \mathbb{Z}^2}{\frac{e^{-it\Omega} -1}{-i\Omega}|n - n_1|\widehat{u_0}({n - n_1 +n_2})\widehat{u_0}({n_1})\widehat{u_0}({n_2})}
$$
where
$$
\Omega = P(n-n_1 + n_2) - P(n) + P(n_1) - P(n_2)
$$
Note that a summand is only nonzero for $6$ possible triplets of $n, n_1, n_2$, since $\widehat{u_0}$ is supported on $\{-m,-(m-1),m+1\}$. We will first consider the triplet that contributes significantly and then show the other five have insignificant contributions. The summands from the other nonlinearities are similar as in \cite{GrandeKurianskiStaffilani2020}.

For the first two terms in the nonlinearity, if $n = m$, $n_1 = -m$, $n_2 = -m+1$ then
$$
P(n - n_1 + n_2) - P(n) + P(n_1) - P(n_2) = -2n \sim m.
$$
We denote the left-hand-side by $\Omega^*$. For the last term, we write 
$$P(-n + n_1 + n_2) + P(n) - P(n_1) - P(n_2) = - \Omega^*.$$
We find the terms that contribute here and call their sum contribution $F$. The terms that contribute here are 
$$
\frac{3i}{2}e^{-itP(n)}{\frac{e^{-it\Omega^*} -1}{-i\Omega^*}(n - n_1 + n_2)\hat{u}({n - n_1})\hat{u}({n_1})\hat{u}({n_2})},
$$ 
$$
\frac{i}{2}e^{-itP(n)}{\frac{e^{-it\Omega^*} -1}{-i\Omega^*}|n - n_1|\hat{u}({n - n_1 + n_2})\hat{u}({n_1})\hat{u}({n_2})},
$$ 
and
$$
\frac{i}{4}e^{-itP(n)}{\frac{e^{it\Omega^*} -1}{i\Omega^*}(n - n_1 - n_2 )\hat{u}({n - n_1 + n_2})\hat{u}({n_1})\hat{u}({n_2})}.
$$
from the terms $\frac{3}{2}|u|^2\partial_x u$, $\frac{i}{2} u |\partial_x| |u|^2$ and $\frac{1}{4}u^2\partial_x u^*$ respectively. 

Choose a small $t$ such that $|t\Omega^*| \ll 1$ (meaning $t \ll m^{-1}$). Then by a Taylor series expansion both the terms 

$$
\frac{e^{-it\Omega^*} -1}{-i\Omega^*} = t + O(t^2).
$$
Since $t$ is small we disregard $O(t^2)$ and the sum of these three contributions is then 
$$
F = \frac{t}{n^{3s}}{\left(\frac{3i}{2}(n+1) + in + \frac{i}{4}(3n+1)\right)} = \frac{t}{n^{3s}}\cdot\frac{i(13n+7)}{4}.
$$
So we find $|F|$ is on the order of 
$$|F| \sim m^{1-3s}t$$ for some small $t$ such that $t \ll m^{-1}$.

For any of the other 5 discrete choices of the triplet $(n, n_1, n_2)$ on which the summand is nonzero, it can be easily shown that $\Omega \gtrsim n^2$. Hence
$$\frac{e^{-it\Omega} -1}{-i\Omega} \lesssim m^{-2},$$ 
and $|n-n_1| \lesssim m$ still must hold so the total contribution of these terms is at most $\lesssim m^{-3s -1}$, which as we will see is not a significant contribution.

Thus $\mathcal{F}_x u_3$ is dominated by $F$ with the peak frequency at $\mathcal{F}u_3(m)$, and we have for small $t \ll m^{-1}$ that
$$
|\mathcal{F}_x u_3(m,t)| \sim |F| \sim n^{-3s+1}{t}
$$ and choosing $t \sim m^{-1}$ we have 
$$
|\mathcal{F}_x u_3(m,t)| \sim m^{-3s}.
$$
But similarly as in \cite{GrandeKurianskiStaffilani2020}, for well-posedness we need
$$1= \|u_0\|_{H^s}\gtrsim \|u_3\|_{L^{\infty}([0, t], H^s)} \gtrsim m^s|\mathcal{F}_x u_3(m,t)| \sim m^{-2s}.$$
Since the number $m$ is arbitrary, there has to hold $s \geq 0$. So if $s<0$ the Dysthe equation must be ill-posed in the Sobolev space $H^s$.
\end{proof}

\section{Acknowledgements}
The authors would like to thank Prof. Gigliola Staffilani from MIT mathematics department for suggesting the project idea. The authors are also grateful to the MIT PRIMES-USA program for the opportunity to research this topic.

\end{spacing}
\end{document}